\documentclass{amsart}
\usepackage{amsmath,amsthm}
\usepackage{amsfonts,amssymb}
\usepackage{enumerate}
\usepackage{accents,color}
\usepackage{graphicx}

\newcommand{\lvt}{\left|\kern-1.35pt\left|\kern-1.3pt\left|}
\newcommand{\rvt}{\right|\kern-1.3pt\right|\kern-1.35pt\right|}

\newtheorem{thm}{Theorem}[section]

\newtheorem{prop}[thm]{Proposition}
\newtheorem{exam}[thm]{Example}

\newtheorem{defn}[thm]{Definition}

\theoremstyle{remark}

 \def\sS{{\mathsf S}}
 \def\sT{{\mathsf T}}

 \def\a{{\alpha}}
 
 \def\g{{\gamma}}

 \def\s{\sigma}
 \def\la{{\langle}}
 \def\ra{{\rangle}}

 \def\CB{{\mathcal B}}
 \def\CD{{\mathcal D}}
 \def\CF{{\mathcal F}}
 \def\CH{{\mathcal H}}
 
 \def\CJ{{\mathcal J}}
 \def\CL{{\mathcal L}}
 
 \def\CP{{\mathcal P}}

 \def\CU{{\mathcal U}}
 \def\CV{{\mathcal V}}
 
 \def\CW{{\mathcal W}}
 \def\BB{{\mathbb B}}
 \def\CC{{\mathbb C}}
 
 \def\NN{{\mathbb N}}
 
 \def\QQ{{\mathbb Q}}
 \def\RR{{\mathbb R}}

\def\lla{\langle{\kern-2.5pt}\langle}
\def\rra{\rangle{\kern-2.5pt}\rangle}

\newcommand{\wt}{\widetilde}
\newcommand{\wh}{\widehat}
\def\rhow{{\lceil\tfrac{\varrho}{2}\rceil}}

\def\ball{\mathbb{B}^d}
\def\sph{\mathbb{S}^{d-1}}
\def\f{\frac}

\graphicspath{{./}}
\begin{document}

\title[Sobolev orthogonal polynomials]
{On Sobolev orthogonal polynomials}

\author{F. Marcell\'an}
\address{Instituto de Ciencias Matem\'aticas (ICMAT) and Departamento de Matem\'aticas, Universidad Carlos III de Madrid\\
Avenida de la Universidad 30, 28911 Legan\'es, Spain}\email{pacomarc@ing.uc3m.es}

\author{Yuan Xu}
\address{Department of Mathematics\\ University of Oregon\\
    Eugene, Oregon 97403-1222.}\email{yuan@uoregon.edu}

\date{\today}
\thanks{The work of the first author has been supported by Direcci\'on General de Investigaci\'on Cient\'ifica y T\'ecnica, Ministerio de Econom\'ia y Competitividad of Spain, grant MTM2012-36732-C03-01.
The work of the second author was supported in part by NSF Grant DMS-1106113}

\begin{abstract}
Sobolev orthogonal polynomials have been studied extensively in the past 20 years. The research in this field has sprawled
into several directions and generates a plethora of publications. This paper contains a survey of the main developments up
to now. The goal is to identify main ideas and developments in the field, which hopefully will lend a structure
to the mountainous publications and help future research.
\end{abstract}

\maketitle
 \tableofcontents

\section{Introduction}
\setcounter{equation}{0}

If $\la \cdot,\cdot \ra$ is an inner product defined on the linear space of polynomials, then orthogonal polynomials
$\{p_n\}_{n \ge 0}$ with respect to the inner product are those polynomials satisfying $\la p_n, p_m \ra =0$
if $n \ne m$. We call them {\it ordinary} orthogonal polynomials if $\la f, g \ra = \int_\RR f(x) g(x) d\mu$ with
respect to a nonnegative Borel measure $d\mu$ supported on an infinite subset of the real line. They are 
called Sobolev orthogonal polynomials when the inner product involves derivatives.

Sobolev orthogonal polynomials were first consider in the early 60's of the last century. In the past 20 years, the
field has seen a rapid development that leads to a large amount of publications. A rough count shows around
four hundred publications in the past two decades. For a new comer, as the second author was at the
beginning of this project, the size of the literature is daunting and worse still is the disarray of the literature.

The theory of ordinary orthogonal polynomials is well established and documented. One essential tool in the
theory is the three-term recurrence relation that orthogonal polynomials satisfy, which holds if the multiplication operator
is symmetric with respect to the inner product, that is, if $\la x p, q \ra = \la p, x q\ra$. In the Sobolev setting, however, 
the multiplication operator is
no longer symmetric and, consequently, the three-term relation no longer holds. The deprival of this fundamental
tool cannot be easily compensated. Different and ad hoc methods have been developed for dealing with different
Sobolev inner products. The result is a theory of Sobolev orthogonal polynomials that appears fragmented and
lack of uniformity.

The purpose of this paper is to provide a survey for the current state of Sobolev orthogonal polynomials. The project
was initiated by the first author, who has worked extensively on the subject, led the second author, who works
in several variables where Sobolev orthogonal polynomials have only been studied recently, through the literature.
After intensive reading and discussion, we decided to trace the development of the subject, identify main ideas and
developments, and provide some structure to the literature so that it can be more accessible for new comers and
researchers interested in this field. The organization of the paper follows roughly the progress of this collective
learning process, the first part of which more or less correlates with the historical development of the field.

We will limit our scope to Sobolev orthogonal polynomials with respect to those inner products that are defined
by integrals with at most finite additional discrete mass points. The main Sobolev inner product that we consider can be written as
$$
    \la f, g \ra = \int_\RR f(x) g(x) d\mu_0 + \sum_{k=1}^m  \int_\RR f^{(k)}(x) g^{(k)}(x) d\mu_k,
$$
where $d\mu_k$, $k=0,1,\ldots, m$, are positive Borel measures on $\RR$. There are essentially three types
of such inner products in this paper:
\begin{enumerate}[\,I.]
\item  $d \mu_0, d\mu_1, \ldots, d\mu_m$ have continuous support.
\item  $d\mu_0$ has continuous support and $d\mu_1, \ldots, d\mu_m$ are supported on
finite subsets.
\item  $d\mu_m$ has continuous support and $d\mu_0,\ldots, d\mu_{m-1}$ are supported
on finite subsets.
\end{enumerate}
In the second and the third cases, we sometimes also consider mixed discrete part, for example,
in the second case,
\begin{align*}
  \la f, g\ra = \int_\RR f(x) g(x) d\mu_0
      +  \left(f(0), f'(0), \ldots, f^{(m)}(0)\right)M \left(g(0), g'(0), \ldots, g^{(m)}(0)\right)^T
\end{align*}
with $M$ being an $(m+1)\times (m+1)$ matrix. We will not consider Sobolev orthogonal polynomials that involve
discrete orthogonal polynomials, $q$-orthogonal polynomials, and complex valued orthogonal polynomials,
since the survey is already long and the ideas and methods used in these settings are often parallel to those
discussed in this paper.

The list of references at the end of the paper is compiled strictly according to the materials covered in the paper.
It is by no means inclusive or complete. We no doubt missed many papers that should be cited and we apologize
to those authors whose work we should have cited.

The paper is organized as follows. After a brief introduction of ordinary orthogonal polynomials in the second section,
we recall the history and early results of Sobolev orthogonal polynomials in the third section. The early results
were essentially established by the method of integration by parts, which is discussed in the fourth section. The
topic went into a long dormant a decade after its beginning and was awakened only when the notion of coherent
pair was introduced. The main idea of coherent pair and its various extensions are expounded in a lengthy fifth section.
Classical orthogonal polynomials, the Jacobi, Laguerre and Hermite polynomials, play an important role in the
development and they are Sobolev orthogonal polynomials themselves for appropriately defined inner products, which
will be explained in the sixth section. Sobolev orthogonal polynomials for the inner products of the second type, those with
derivatives appear only in the point evaluations, are discussed in the seventh section. Some of these polynomials
satisfy differential equations, which and other results concerning differential equations are presented in the eighth
section. Two important properties of Sobolev orthogonal polynomials, zeros and asymptotics, are addressed in the
ninth and tenth sections, respectively.  More recently, prompted partly by problems in numerical solution
of partial differential equations, Sobolev orthogonal polynomials in several variables have come into being. What is
known up to date in this direction is reported in the eleventh section. Finally, in the twelfth section, we
discuss Fourier expansions in Sobolev orthogonal polynomials. The lack of Christoffel-Darboux formula for Sobolev
orthogonal polynomials, consequence of the lack of three-term recurrence relation, deprives an important tool for studying
convergence and summability of Fourier orthogonal expansions. As a consequence, except for certain inner product
of the second type and some numerical experiments, the convergence of Fourier expansions in Sobolev orthogonal
polynomials has not been resolved. We consider this deficiency one of the major open problems that deserves
to be studied intensively. This call of action seems a fitting point to end our survey.

\section{Orthogonal polynomials}
\setcounter{equation}{0}

In this section, we introduce notation and basic background concerning the general structure of orthogonal polynomials.
Although the results in this section are mostly classical, it is necessary to fix notations and recall results
that will be essential in our discussion.

\subsection{General properties of orthogonal polynomials}
All functions encounter in this paper are real valued. Let $\Pi$ denote the linear space of polynomials with real coefficients 
on the real
line and, for $n = 0,1,\ldots$, let $\Pi_n$ denote the linear subspace of polynomials of degree at most $n$.

We consider orthogonal polynomials on the real line. Let $\la \cdot,\cdot \ra$ be a symmetric bilinear form defined on
$\Pi \times \Pi$. It is an inner product if $\la p, p \ra > 0$ for all nonzero polynomial $p \in \Pi$. A sequence
of polynomials $\{P_n \}_{n\geq0}$ is called orthogonal with respect to
$\la \cdot,\cdot\ra$ if $\deg P_n = n$ and
$$
      \la  P_n, P_m \ra  = 0, \qquad n \ne m.
$$
$P_n$ is said to be monic if $P_n (x) = x^n + a_{n,n-1} x^{n-1}+ \cdots$. For $n =0,1,2,\ldots,$ let
$$
M_n :=
\left[ \begin{matrix}
  \la 1,1\ra & \la 1,x\ra & \cdots &\la 1, x^n \ra \\
  \la  x,1\ra & \la x, x \ra & \cdots &\la x, x^n \ra \\
  \cdots & \cdots &  \ddots & \cdots \\
 \la  x^n,1\ra & \la x^n, x\ra & \cdots & \la x^n, x^n \ra
\end{matrix} \right].
$$
If $\la \cdot,\cdot \ra$ is an inner product, then $M_n$ is positive definite i.e. $\det M_n > 0$ for every $n \in \NN_0.$
If $\det M_n \ne 0$ for all $n \in \NN_0$, then a sequence of monic orthogonal polynomials exists. In fact, the monic
orthogonal polynomials are $P_0(x) =1$ and, for $n \ge 1$,
\begin{equation}\label{eq:detPn}
P_n (x) = \frac{1}{\det M_{n-1}} \det \left[ \begin{array}{c|c}
  M_{n-1} & \begin{matrix} \la 1, x^{n} \ra  \\ \la x, x^{n} \ra \\ \vdots \\ \la x^{n-1}, x^{n} \ra \end{matrix}\\
  \hline 1, x, \hdots, x^{n-1} & x^n
      \end{array} \right].
\end{equation}

Let $\mu$ be a positive Borel measure supported on the real line such that all its moments, $\int_\RR x^n d\mu$
with $n =0,1,2,\ldots$, are finite. For such a measure,
$$
  \la f, g \ra_{d\mu}: = \int_{\RR}  f(x) g(x) d\mu
$$
defines an inner product. We can define bilinear forms from linear functionals. Indeed, let $\CU: \Pi \mapsto \RR$
be a linear functional and denote its action on $p \in \Pi$ by $\la \CU, p\ra$.  Associated to $\CU$ we define a
bilinear form
$$
    \la f, g  \ra :=  \la \CU, f g \ra.
$$

\begin{defn}
The linear functional $\CU$ is called quasi-definite if $\det M_n \ne 0$ for all $n \in \NN_0$, and it is called
positive definite if $\det M_n > 0$ for all $n\in \NN_0$.
\end{defn}

When $\CU$ is quasi-definite, orthogonal polynomials with respect to the bilinear form defined by $\CU$ exist,
which shall be called orthogonal polynomials with respect to $\CU$. When $\CU$ is positive definite, the bilinear
form is an inner product given by $\la \cdot,\cdot \ra_{d\mu}$ for a positive Borel measure $\mu.$

Let $\partial$ denote the derivative $\partial f(x) := f'(x)$ and let $q$ be a fixed polynomial in $\Pi$.
For a linear functional $\CU$, the linear functionals $\partial \CU$ and $q \CU$ are defined, respectively, by
$$
  \la  \partial \CU, p \ra := - \la \CU, \partial p \ra \quad \hbox{and} \quad \la q\CU, p\ra: = \la \CU, q p\ra,
    \qquad \forall p \in \Pi.
$$
For $a \in \RR$, the delta functional $\delta_a$ is defined by $\la \delta_a, p \ra = p(a)$ for all $p \in \Pi$.

For a quasi-definite linear functional $\CU$, monic orthogonal polynomials $\{P_n\}_{n\geq0}$ are characterized by
the three-term recurrence relation
$$
   x P_n(x) = P_{n+1}(x) + b_n P_n(x) + c_n P_{n-1}(x), \qquad n\geq0
$$
where $c_n \ne 0$ for $n\geq 1$. The three-term recurrence relation plays an important role in the study of 
ordinary orthogonal polynomials. It is equivalent, in particular, to the Christoffel-Darboux formula,
$$
   \sum_{k=0}^n \frac{P_k(x)P_k(y)}{h_k} =  \frac{1}{h_n} \frac{P_{n+1}(x) P_n(y) - P_{n+1}(y) P_n(x)}{x-y},
$$
where $h_k = \sqrt{\la P_k,P_k \ra}$. For orthogonal polynomials with
respect to an inner product, the three-term relation holds if an only if $\la \cdot, \cdot \ra$ satisfies
$$
    \la x p, q\ra = \la p, x q\ra, \quad  p, q \in \Pi.
$$
i.e. the multiplication operator is a symmetric operator with respect to the above inner product. This property, however,
does not hold for Sobolev inner product in general as we will show in the sequel.

\subsection{Classical orthogonal polynomials} These polynomials are associated with the following weight functions
\begin{enumerate}[\quad \rm (1)]
\item Hermite: $w(x) = e^{-x^2}$ on $(-\infty, \infty)$;
\item Laguerre: $w_\a(x) = x^\a e^{-x}$ on $[0, \infty)$, $-\a \notin \NN$;
\item Jacobi: $w_{\a,\beta}(x) = (1-x)^\a (1+x)^\beta$ on $(-1, 1)$, $-\a \notin \NN, - \beta\notin \NN, -\a-\beta \notin \NN$.
\end{enumerate}
We denote the corresponding linear functional by $\CH$, $\CL_\a$, and $\CJ_{\a,\beta}$, respectively.
These linear functionals are quasi-definite for all ranges of their parameters. Moreover, $\CH$ is positive definite,
$\CL_\a$ is positive definite if $\a > -1$ and $\CJ_{\a,\beta}$ is positive definite if $\a > -1$ and $\beta > -1$.

In the positive definite case, the orthogonal polynomials for these weight functions are called classical. They are
the Hermite polynomials $H_n$, the Laguerre polynomials $L_n^{(\a)}$, and the Jacobi polynomials $P_n^{(\a,\beta)}$,
which are defined in terms of hypergeometric functions as follows,
\begin{align*}
  H_n(x)&  = (2 x)^n {}_{2}F_{0} \left( \begin{matrix} - \frac{n}{2}, - \frac{n+1}{2} \\
      -  \end{matrix};\frac{1}{x^2}\right),  \\
  L_n^{(\a)}(x) &  =\frac{(\alpha +1)_{n}}{n!} {}_{1}F_{1} \left( \begin{matrix} - n \\
        \a+1 \end{matrix};x \right),  \\
  P_{n}^{(\alpha ,\beta )}(x)  &=\frac{(\alpha +1)_{n}}{n!}{}_{2}F_{1}\left( \begin{matrix} -n,
n+\alpha +\beta +1\cr
\alpha +1\end{matrix};\frac{1-x}{2}\right).
\end{align*}
The Gegenbauer polynomial $C_n^\lambda$ is a constant multiple of the Jacobi polynomial
$$
C_n^\lambda (x)  =   \frac{(2 \lambda)_n}{ (\lambda+\f12)_n} P_n^{(\lambda-\f12, \lambda- \f12)}(x), \qquad \lambda > - \f 12.
$$

There are several characterizations of classical orthogonal polynomials. If we consider only orthogonal
polynomials with respect to a real inner product, then polynomials in each classical family are eigenfunctions of
a second order linear differential operator with polynomial coefficients; moreover, up to a real linear change of
variables,
they are the only orthogonal polynomials that satisfy this property. Let $\CU$ be one of the
classical linear functionals. Then $\CU$ satisfies the {\it Pearson equation}
$$
   \partial (\phi \CU) = \psi \CU, \qquad \deg \phi \le 2 \quad \hbox{and}\quad \deg \psi =1;
$$
more precisely, for $- \a \notin \NN$, $-\beta \notin \NN,$ and $-\a-\beta \notin \NN$,
$$
  \partial \CH = - 2 x \CH, \quad \partial( x \CL_{\a}) = (-x+\a+1) \CL_{\a}, \quad
     \partial ( (x^{2}-1) \CJ_{\a,\beta}) = ( (\a+\beta+2)x  + \a- \beta) \CJ_{\a,\beta}.
$$
Furthermore, up to a linear change of variables, the only other family of linear functionals that satisfy the
Pearson equation corresponds to the Bessel polynomials, and the linear functional is defined by an
integral along a curve of the complex plane surrounding the origin and complex weight
$$
w_{\alpha}(z)= \sum_{k=0}^{\infty} \frac{1}{(\alpha +1)_{k}} (-2/z)^{k},
$$
where $(a)_{n} = a (a+1) ...(a+n-1), n\geq 1,$ and $(a)_{0}=1$ denotes the Pochhammer symbol (see \cite{Chi78}).


\section{History and earlier results}
\setcounter{equation}{0}

The starting point of the Sobolev orthogonal polynomials can be traced back to the paper \cite{Lewis} by Lewis, who
asked the following question: Let $\a_0, \ldots, \a_p$ be monotonic, non-decreasing functions defined on $[a,b]$
and let $f$ be a function on $[a,b]$ that satisfies certain regularity conditions. Determine a polynomial $P_n$ of
degree $\le n$  that minimizes
$$
     \sum_{k=0}^p \int_a^b | f^{(k)} (x) - P_{n}^{(k)}(x) |^2 d\a_k(x).
$$
Lewis did not use Sobolev orthogonal polynomials and gave a formula for the reminder term of the approximation as
an integral of the Peano kernel. The first paper on Sobolev orthogonal polynomials was published by Althammer \cite{Alt},
who attributed his motivation to Lewis's paper.

The Sobolev orthogonal polynomials considered in \cite{Alt} are orthogonal with respect to the inner product
\begin{equation} \label{eq:ipdLeg1}
  \la f, g\ra_S = \int_{-1}^1 f(x)g(x) dx + \lambda \int_{-1}^1 f'(x) g'(x) dx, \quad \lambda > 0.
\end{equation}
These Sobolev-Legendre polynomials were systematically studied in \cite{Alt}. Some of the results in
\cite{Alt} were simplified and further extended by Sch\"afke in \cite{Sch72}. These earlier works already
demonstrated several characteristic features of Sobolev orthogonal polynomials. Let $S_n(\cdot;\lambda)$ denote
the orthogonal polynomial of degree $n$ with respect to the inner product $\la \cdot,\cdot \ra_S$, normalized
so that $S_n(1; \lambda) =1$, and let $P_n$ denote the $n$-th Legendre polynomial. The following properties hold for
$S_n(\cdot;\lambda)$:
\begin{enumerate}[\quad \rm (1)]
 \item $\{S_n(\cdot; \lambda)\}_{n \ge 0}$ satisfies a differential equation
  $$
       \lambda S_n''(x;\lambda) - S_n(x;\lambda) = A_n  P_{n+1}'(x) + B_n  P_{n-1}'(x),
  $$
where $A_n$ and $B_n$ are constants that can be given by explicit formulas.
\item $\{S_n(\cdot; \lambda)\}_{n \ge 0}$ satisfies a recursive relation
 $$
     S_n(x;\lambda) - S_{n-2}(x; \lambda) = a_n (P_n(x) - P_{n-2}(x)), \quad n =1,2,\ldots.
 $$
\item $S_n(\cdot;\lambda)$ has $n$ real simple zeros in $(-1,1)$,
\end{enumerate}
For a more detailed account on the development of these results,
we refer to the original articles or to \cite{Mei96}, which contains a nice survey of early history of Sobolev orthogonal
 polynomials.  The Sobolev-Legendre polynomials
were also studied by Gr\"obner, who established a version of the Rodrigues formula for in \cite{Gr67}, which states
that, up to a constant factor $c_n$,
$$
 S_n(x;\lambda) = c_n \frac{\partial^{n}}{1- \lambda \partial^{2}} \left((x^{2}-x)^{n} - \alpha_{n}( x^{2}- x)^{n-1}\right )
$$
where $\alpha_{n}$ are real numbers explicitly given in terms of $\lambda$ and $n$.
Furthermore, in \cite{Coh75}, Cohen proved that the zeros of $S_n(\cdot;\lambda)$ interlace with those of the Legendre
polynomial $P_{n-1}$ if $\lambda \ge 2/n$, and he also established the sign of the connecting coefficients
of $S_n(\cdot;\lambda)$ expanded in terms of the sequence $\{S_m(\cdot;\mu)\}_{m\ge 0}$ for $\lambda \ne \mu$.

In \cite{Alt}, Althammer also gave an example in which he replaced $dx$ in the second integral in $\la \cdot,\cdot\ra_S$
by $w(x) dx$ with $w(x) = 10$ for $-1 \le x \le 0$ and $w(x) = 1$ for $0 \le x \le 1$, and made the observation that
$S_2(x;\lambda)$ for this new inner product has one real zero outside of $(-1,1)$.

Another earlier paper is \cite{Br72}, in which Brenner considered the inner product
$$
  \la f, g\ra: = \int_0^\infty f(x) g(x) e^{-x} dx + \lambda \int_{0}^\infty f'(x) g'(x) e^{-x} dx, \quad \lambda > 0,
$$
and obtained results similar to those of Althammer.

An important paper in the early development of the Sobolev orthogonal polynomials is \cite{SW73}, in which
Sch\"afke and Wolf considered a family of inner products
\begin{equation} \label{eq:SW-ipd}
 \la f ,g \ra_S = \sum_{j,k = 0}^\infty \int_a^b f^{(j)}(x) g^{(k)}(x) v_{j,k}(x) w(x) dx,
\end{equation}
where $w$ and $(a,b)$ are one of the three classical cases, Hermite, Laguerre and Jacobi, and
the functions $v_{j,k}$ are polynomials that satisfy $v_{j,k} = v_{k,j}$, $j,k = 0,1,2,\ldots$, and permit writing
the inner product \eqref{eq:SW-ipd} as
$$
     \la f ,g \ra_S = \int_a^b f(x) \CB g(x) w(x) dx \quad \hbox{with}\quad
      \CB g: = w^{-1} \sum_{j,k=0}^\infty (-1)^j \partial^j (w v_{j,k} \partial^k )g
$$
through integration by parts. Under further restrictions on $v_{j,k}$, they narrowed down to eight classes of Sobolev orthogonal
polynomials, which they call simple generalization of classical orthogonal polynomials. These cases are given by

\begin{enumerate}[\qquad   ]
\item[A1.] $v_{j,k} =0,$ with $|j-k|>1,$ $v_{j,j} = (1-x^{2})^{j} (b_{j} x + c_{j}),$ $v_{j,j-1} = d_{j} (1-x^{2})^{j}.$

\item[A2.] $v_{j,k} =0,$ with $|j-k|>2,$ $v_{j,j} = (1-x^{2})^{j} (b_{j} (1- x^{2}) + c_{j}),$ $v_{j,j-1} = d_{j}x (1-x^{2})^{j}, j\geq 1,$ $v_{j,j-2} = e_{j} (1-x^{2})^{j}, j\geq 2.$

\item[A3.] $v_{j,k} =0,$ with $|j-k|>1,$ $v_{0,0}= c_{0},$ $v_{j,j} = (1+ x)^{j} (1- x)^{j-1}(b_{j} x + c_{j}),$ $v_{j,j-1} = d_{j} (1+ x)^{j} (1- x)^{j-1}, j\geq 1 .$

\item[A4.] $v_{j,k} =0,$ with $|j-k|>2$, $v_{0,0}= c_{0},$ $v_{j,j} = (1-x^{2})^{j-1} (b_{j} (1- x^{2}) + c_{j}),$ $v_{j,j-1} = d_{j}x (1-x^{2})^{j-1}, j\geq 1,$ $v_{j,j-2} = e_{j} (1-x^{2})^{j-1}, j\geq 2.$

\item[B1.] $v_{j,k} =0,$ with $|j-k |>1$, $v_{j,j} = x^{j} (b_{j} x + c_{j}),$ $v_{j,j-1} = d_{j} x^{j}.$

\item[B2.] $v_{j,k} =0,$ with $|j-k \, | >1$, $v_{0,0}= c_{0},$ $v_{j,j} = x^{j-1} (b_{j} x + c_{j}),$ $v_{j,j-1} = d_{j} x^{j-1}, j\geq 1.$

\item[B3.] $v_{j,k} =0,$ with $|j-k |>1$, $v_{j,j} = b_{j} x + c_{j},$ $v_{j,j-1} = d_{j}.$

\item[C1.] $v_{j,k} =0,$ with $|j-k|>2,$ $v_{j,j} = b_{j} x^{2} + c_{j},$ $v_{j,j-1} = d_{j}x, j\geq 1,$ $v_{j,j-2} = e_{j}.$
\end{enumerate}
The main results in \cite{SW73} extended all earlier results on Sobolev orthogonal polynomials. Let $\{S_n\}_{n\geq0}$ denote a
sequence of Sobolev orthogonal polynomials with respect to $\la \cdot,\cdot \ra_S$ in \eqref{eq:SW-ipd} and let $\{T_n\}_{n\geq0}$ denote a sequence of ordinary orthogonal polynomials with respect to the inner product
$$
   \la f, g\ra = \int_a^b f(x) g(x) u(x) w(x) dx,
$$
where $u$ is a polynomial of degree at most 2, which is known explicitly in each class. It was shown in \cite{SW73} that,
under appropriate normalizations of $S_n$ and $T_n$, there exists a sequence of real numbers  $\{a_n\}_{n\geq0}$ such that
$$
   S_{n+r} - S_n = a_{n+r} u T_{n+r - k},
$$
where $r = 1$ or $r$ depending on the case, $k$ equals the degree of $u$ and $a_{n+r}$ is a constant, and the
differential operator $\CB$ satisfies
$$
    \CB S_n = b_n (a_n T_{n + s } - a_{n+r} T_{n+s - r}),
$$
where $s = 0, 1$, or $2$, depending on the class. Furthermore, sufficient conditions were given in \cite{SW73}
for $S_n$ to have all simple zeros in $(a,b)$. The paper, however, is not easy to digest. Except the theorem on zeros,
results were not stated in theorems and it is no small task to figure out the exact statement for each individual class.

The primary tool in the early study of Sobolev orthogonal polynomials is integration by parts. Sch\"afke and Wolf
explored when this tool is applicable and outlined potential Sobolev inner products. It is remarkable
that their work appeared in such an early stage of the development of Sobolev orthogonal polynomials. However, instead
of stirred to action by \cite{SW73}, the study of Sobolev orthogonal polynomials unexpectedly became largely dormant
for nearly two decades, from which it reemerged only when a new ingredient,
{\it coherent pairs}, was introduce by Iserles, Koch, N{{\o}}rsett and Sanz-Serna in \cite{AKNS}.

\section{Method of integration by parts}
\setcounter{equation}{0}

In this section, we explain the method of integration by parts in the study of Sobolev orthogonal polynomials by
considering the inner product that involves only first order derivative,
\begin{equation} \label{eq:sec4-ipd}
  \la f, g\ra_S: =\la f, g \ra + \lambda \la f', g'\ra  \quad \hbox{with} \quad \la f,g\ra: = \int_a^b f(x) g(x) d\mu(x),
\end{equation}
where we assume that $d\mu = u(x) dx$ and $u$ satisfies the relation
\begin{equation} \label{eq:phi-psi}
\partial (\phi(x) u(x) ) = \psi(x) u(x),
\end{equation}
in which $\phi$ and $\psi$ are fixed polynomials, with $\phi$ monic and deg$\psi \geq 1,$ and we assume that $\phi$ or $u$ vanish on the end points of the interval $(a,b)$, under limit if $a = -\infty$ or $b = \infty$, so that integration by parts can be carried out.

In the case of the Laguerre weight function $w_\a(x) = x^\a e^{-x}$, $\phi(x) = x$ and $\psi(x) =\a +1- x$. In the case
of the Gegenbauer weight function $w_\lambda(x) = (1-x^2)^{\lambda-\f12}$, $\phi(x) = (1-x^2) $ and 
$\psi(x) = - (2\lambda+1)x$.

Associated with the inner product \eqref{eq:sec4-ipd}, the following differential operator is useful,
$$
     \CF : = \phi(x) I - \lambda \CF_0, \qquad \CF_0: =[u(x)]^{-1}  \phi(x) \partial \left[ u(x) \partial \right],
$$
where $I$ denotes the identity operator. Applying \eqref{eq:phi-psi} allows us to write $\CF_0$ as
$$
   \CF_0 = \phi(x)  \partial^2  + [\psi(x)- \phi'(x)] \partial .
$$

\begin{prop}
Assume that $u(x) \CF_0 $ is zero when $x=a$ and $x=b$. Then, for $f, g\in \Pi$,
\begin{enumerate}[\quad \rm (1)]
 \item $\la \phi f, g \ra_S = \la f, \CF g\ra$.
 \item $\CF$ is self-adjoint; that is, $\la \CF f, g \ra_S = \la f, \CF g\ra_S$.
\end{enumerate}
\end{prop}

\begin{proof}
Integration by parts shows immediately that $\la (\phi f)', g' \ra = - \la \phi f, \CF_0 g\ra$, from which (1) follows readily.
Furthermore, it shows that
$$
   \la (\CF_0 f)' , g' \ra = - \int_a^b  \CF_0 f(x) (g'(x) u(x))' dx = - \int_a^b (f'(x) u(x))' \CF_0 g(x) dx = \la f', (\CF_0 g)' \ra,
$$
where the integration by parts is justified since $u (x)\CF_0$ is zero when $x = a$ or $x =b$, which can then be used to
verify (2).
\end{proof}

Let again $\{S_n(\cdot; \lambda)\}_{n\geq0}$ be the sequence of monic Sobolev polynomials and let $\{P_n\}_{n\geq0}$ be the sequence of monic polynomials orthogonal with respect to $\la \cdot, \cdot \ra.$

\begin{prop} \label{prop:LaguerreSOP}
Assume that $\phi$ is of degree $s$ and $\psi$ is of degree at most $s-1$. Then,
for $n \ge s$,
\begin{enumerate}[\quad \rm (i)]
 \item $\phi(x) P_n(x, d\mu) = S_{n+s} (x;\lambda) +\displaystyle{ \sum_{j = n-s}^{n+s-1} a_{j,n}(\lambda) S_{n+j}(x;\lambda)}$,
 \item $\CF S_n(x;\lambda) =  S_{n+s} (x;\lambda) + \displaystyle{ \sum_{j = n-s}^{n+s-1} b_{j,n}(\lambda) S_{n+j}(x;\lambda)}$.
\end{enumerate}
\end{prop}

\begin{proof}
The additional assumptions on $\phi$ and $\psi$ show that $\CF_0: \Pi_n \mapsto \Pi_{n+s-2}$, so that
$\CF: \Pi_n \to \Pi_{n+s}$. By (1) of the previous proposition, $\la \phi P_n(\cdot; d\mu), g \ra_S = \la P_n(\cdot; d\mu),
\CF g\ra = 0$ if $\CF g \in \Pi_{n-s}$, which implies (i). The same argument proves (ii) as well.
\end{proof}

The relation in (i) is the recursive relation and (ii) is the difference-differential equation satisfied by the Sobolev
orthogonal polynomials. The constants $a_{j,n}$ and $b_{j,n}$ can be explicitly determined in the case of
the classical Laguerre weight $w_\a$ and the Gegenbauer weight $w_\lambda$.

Take the Laguerre weight $w_\a$ as an example. The monic Laguerre polynomial of degree $n$ is
$P_n(x) := n!(-1)^n L_n^{(\a)}(x)$. By the property of the Laguerre polynomial and its derivative, we have
\begin{equation}\label{eq:Laguerre-coh}
    Q_n'(x) = n P_{n-1}(x), \qquad \hbox{where} \quad Q_n (x):= P_n (x) + n P_{n-1}(x).
\end{equation}
It follows readily that $\la Q_n, q\ra_S = 0$ if $q \in \Pi_{n-2}$, which implies immediately that
\begin{equation} \label{eq:Laguerre-SOP}
     Q_n(x) = S_n(x; \lambda) + d_{n-1}(\lambda) S_{n-1}(x;\lambda)
\end{equation}
for some constant $d_{n-1}(\lambda)$. Both the differential relation and the recursive relation in
Proposition \ref{prop:LaguerreSOP} can be made explicit using the constant $d_n(\lambda)$. For
example, using the three-term relation for the monic Laguerre polynomial and
\eqref{eq:Laguerre-SOP}, the recurrence relation (i) becomes
$$
    x P_n(x) =  S_{n+1}(x;\lambda) +  [n+\a +  d_n(\lambda)]S_n(x;\lambda) + [(n+\a)d_{n-1}(\lambda)] S_{n-1}(\lambda;x).
$$
Finally, the value of $d_{n-1}(\lambda)$ can be deduced recursively.

\begin{prop}
For $n =2, 3, \ldots,$
$$
  d_n(\lambda) =  \frac{(n+1)(n+\a)}{ (\lambda+2)n + \a - d_{n-1}(\lambda)} \quad \hbox{with}\quad
      d_1(\lambda) := \frac{2(\a+1)}{ \lambda+\a+1}.
$$
\end{prop}

\begin{proof}
Using the relation \eqref{eq:Laguerre-SOP} for both $S_{n+1}(\cdot;\lambda)$ and $S_n(\cdot;\lambda)$, it follows that
\begin{align*}
 0 = \la S_{n+1}, S_{n} \ra_S = \la Q_{n+1}, Q_{n} \ra_S - d_n(\lambda) \la Q_{n}, Q_{n} \ra_S
   + d_n(\lambda) d_{n-1}(\lambda) \la Q_{n-1}, Q_{n} \ra_S.
\end{align*}
Evaluating the inner products in the right hand side by using \eqref{eq:Laguerre-coh}, this gives the
stated recursive relation.
\end{proof}

The recurrence relation shows that $d_n(\lambda)$ is a rational function of $\lambda$. In particular, if we express
it as $d_n(\lambda) = (n+1)(n+\lambda) q_{n-1}(\lambda) /q_n(\lambda)$, then $q_n(\lambda)$ satisfies a three-term recurrence relation and $q_n(\lambda)$
is related to the Pollaczek polynomials.

The monic Gegenbauer polynomial $P_n^\lambda(x)$ satisfies the relation
$$
   P_n^\lambda(x) =  \frac{1}{n+1} \frac{d}{dx}P_{n+1}^\lambda(x) - \frac{\lambda-1}{4 (n+\lambda-1)(n+\lambda)}\frac{d}{d x} P_{n-1}^\lambda (x).
$$
Using the parity of the Gegenbauer polynomials, the above procedure for the Laguerre weight can
be carried for the Sobolev orthogonal polynomials associated with the Gegenbauer weight $w_\lambda$.

The above analysis shows how the classical work of Althammer for the Legendre weight and Brenner
for $e^{-x}$, the Laguerre weight with $\a =0$, can be worked out. This streamlined analysis was
carried out in \cite{MPP95, MPP96}, which already incorporated the idea of the coherent pair that
will be addressed in the next section. The method of integration by parts is applicable to Sobolev inner products involving derivatives
of higher order, which gives an indication how the general setting of Sch\"afke and Wolf in  \cite{SW73}
can be developed. It is no longer clear, however, if the finer result that hold for the Laguerre and the
Gegenbauer weights can be established in the general setting of \cite{SW73}.

\section{Coherent pairs}
\setcounter{equation}{0}

The notion of coherent pair was first introduced in \cite{AKNS} and it has become an important tool and a source of
new development. Its appearance coincides with the revival of the field of Sobolev orthogonal polynomials,
which has flourished ever since.  This section explains this notion and its various generalizations.

\subsection{Coherent pairs}
The coherent pair introduced in \cite{AKNS} is defined for the inner product
\begin{equation} \label{eq:cohernt-ipd}
  \la f, g \ra_\lambda = \int_a^b f(x) g(x) d \mu_0(x) + \lambda \int_a^b f'(x) g'(x) d\mu_1(x),
\end{equation}
where $- \infty \le a < b \le \infty$, $\mu_0$ and $\mu_1$ are positive Borel measures on the real
line with finite moments of all orders. Let $P_n(\cdot; d\mu_i)$ denote the monic orthogonal polynomial of
degree $n$ with respect to $d\mu_i$.

\begin{defn}
The pair $\{d\mu_0, d\mu_1\}$ is called coherent if there exists a sequence of nonzero
real numbers $\{a_{n}\}_{n\geq1}$ such that
\begin{equation} \label{eq:coh-A1}
    P_n(\cdot; d\mu_1) = \frac{P_{n+1}'(\cdot;d\mu_0)}{n+1} + a_n \frac{P_n'(\cdot;d\mu_0)}{n}, \qquad n \ge 1.
\end{equation}
If $[a,b]= [-c,c]$ and $d\mu_0$ and $d\mu_1$ are both even, then $\{d\mu_0, d\mu_1\}$
is called a symmetrically coherent pair if
\begin{equation} \label{eq:coh-A2}
   P_n(\cdot;d\mu_1) = \frac{ P_{n+1}'(\cdot;d\mu_0)}{n+1} + a_n \frac{P_{n-1}'(\cdot;d\mu_0)}{n-1}, \quad n \ge 2.
\end{equation}
In the case of $d\mu_1 = d\mu_0$, we call $d \mu_0$ self-coherent.
\end{defn}

For the classical orthogonal polynomials, it is easy to see that the following examples are
coherent pairs:

\begin{exam} The Laguerre measure $d\mu = x^\a e^{-x} dx$, $\a > -1$, is an example of
a self-coherent pair. The Gegenbauer measure $d\mu = (1-x^2)^{\mu-1/2}$, $\lambda > 0$, is an example of
a symmetrically self-coherent pair. The Jacobi weight yields a natural coherent pair; for $\a, \beta > -1$
$$
  d\mu_0 = (1-x)^\a(1+x)^\beta \quad \hbox{and} \quad  d\mu_1 = (1-x)^{1+\a}(1+x)^{1+  \beta}.
$$
\end{exam}

Let $\{S_n(\cdot;\lambda)\}_{n\geq 0}$ denote the sequence of monic Sobolev orthogonal polynomials with respect to
$\la \cdot, \cdot \ra_\lambda$. Then $S_n(\cdot;\lambda)$ is given
by the determinant expression \eqref{eq:detPn}. It is easy to see that
$$
      \Lambda_n(x) : = \lim_{\lambda \to \infty} S_n(x;\lambda)
$$
defines a monic polynomial $\Lambda_n(x)$ of degree $n$ which satisfies the following properties
\begin{equation} \label{eq:limSn}
\Lambda_n'(x) = n P_{n-1}(x;d\mu_1) \quad \hbox{and}\quad
\int_\RR \Lambda_n(x) d\mu_0 = 0, \quad n \ge 1.
\end{equation}

\begin{thm} \label{thm:SOP-OP-A1}
If $\{d\mu_0, d\mu_1\}$ is a coherent pair, then
\begin{equation} \label{eq:SOP-OP-A1}
    S_n(x;\lambda) + b_{n-1}(\lambda) S_{n-1}(x;\lambda)=  P_n(x;d\mu_0)  + \wh a_{n-1}  P_{n-1}(x;d\mu_0),
\end{equation}
where $\wh a_{n-1} = n a_n /(n-1)$ and  $b_{n-1}(\lambda) =\wh a_{n-1} \|P_{n-1}(\cdot; d\mu_0)\|_{d\mu_0}^2 /{\|S_{n-1}(\cdot;\lambda)\|_\lambda^2}$.
\end{thm}

\begin{proof}
By \eqref{eq:coh-A1} and \eqref{eq:limSn}, we see that
\begin{equation} \label{eq:Ldn-Pn}
   \Lambda_n(x) = P_n(x;d\mu_0) + \wh a_{n-1} P_{n-1}(x;d\mu_0).
\end{equation}
For $0 \le j \le n-2$, it follows from \eqref{eq:limSn} that
$$
\la \Lambda_n, S_j(\cdot;\lambda) \ra_\lambda = \la \Lambda_n, S_j(\cdot;\lambda) \ra_{d\mu_0} +
    n \la P_{n-1}(\cdot;d\mu_1)  S'_j(\cdot;\lambda) \ra_{d\mu_1} =0,
$$
which implies, considering the expansion of $\Lambda_n$ in $S_j(\cdot;\lambda)$, that
$$
 \Lambda_n(x) = S_n(x;\lambda) + b_{n-1}(\lambda) S_{n-1}(x;\lambda), \quad \hbox{where}\quad b_{n-1}(\lambda)
      = \frac{\la \Lambda_n, S_{n-1}(\cdot;\lambda)\ra_\lambda}{\|S_{n-1}(\cdot;\lambda)\|_\lambda^2}.
$$
The formula of $b_{n-1}(\lambda)$ follows from $\la \Lambda_n', S_{n-1}'(\cdot;\lambda)\ra_{d\mu_1} =0$ as well as from the fact
that both $P_{n-1}(\cdot; d\mu_0)$ and $S_{n-1}(\cdot;\lambda)$ are monic.
\end{proof}

Since $b_{n-1}(\lambda)$ depends only on $S_{n-1}(\cdot;\lambda)$, the identity  \eqref{eq:SOP-OP-A1} can
be used to compute $S_n(\cdot; \lambda)$ inductively. Furthermore, it implies the expansion
$$
 S_n(x;\lambda) = P_n(x;d\mu_0) + \sum_{k=0}^{n-1} \Big(\prod_{j=k+1}^{n-1} b_j(\lambda) \Big) (\wh a_k- b_k(\lambda))  P_k(x;d\mu_0),
$$
where we adopt the convention that $\prod_{j=n}^{n-1} b_j(\lambda) =1$.

In \cite{AKNS}, a different normalization of Sobolev orthogonal polynomials is chosen. Let $P_n = P_n(\cdot; d\mu_0)$. Define
$$
  \wt S_n(x;\lambda): =\frac{1}{ \displaystyle{\prod_{k=1}^{n-1} }\|P_k\|^2_{d\mu_0} }
  \det \left[ \begin{matrix} \la P_1, P_1\ra_\lambda & \la P_1, P_2\ra_\lambda & \ldots &  \la P_1, P_n\ra_\lambda \\
   \la P_2, P_1\ra_\lambda & \la P_2, P_2\ra_\lambda & \ldots &  \la P_2, P_n\ra_\lambda \\
   \ldots & \ldots & \ddots & \ldots \\
   \la P_{n-1}, P_1\ra_\lambda & \la P_{n-1}, P_2\ra_\lambda & \ldots &  \la P_{n-1}, P_n\ra_\lambda \\
    P_1(x) & P_2(x) & \ldots & P_n(x)
  \end{matrix} \right].
$$
With this normalization, the identity \eqref{eq:SOP-OP-A1} becomes
$$
  \wt S_{n}(x;\lambda) - \wt S_{n-1}(x;\lambda) = a_n(\lambda) (P_{n}(x; d\mu_0)- P_{n-1}(x; d\mu_0))
$$
and the coefficients $\a_k(\lambda)$ in the expansion
$$
  \wt S_n(x;\lambda) = \sum_{k=1}^{n-1} \a_{k}(\lambda) P_k(x; d\mu_0) + a_n(\lambda)  P_{n}(x; d\mu_0)
$$
depend only on $k$. Moreover, $\a_k(\lambda)$ is a polynomial of degree $k$ in the variable $\lambda$ such that $\a_k(0)=0$.
Furthermore, let $R_k(\lambda):= \a_{k+1}(\lambda)/\lambda$; then the sequence of polynomials $\{R_k \}_{k\geq0}$, for which $\deg R_{k}=k$,
satisfies a three-term relation and, as a consequence, is a sequence of
orthogonal polynomials with respect to a positive
Borel measure. In the case of the Laguerre weight $d\mu_0 (x)  = d \mu_1(x) = x^\a e^{-x} dx$ on
$[0, \infty)$, the $R_k(\lambda)$ are, up to a multiple constant, Pollaczek polynomials.

In \cite{Mei93}, it was observed that the zeros of $\Lambda_n$ interlace with those of $P_{n-1}(\cdot;d\mu_0)$ and
with those of $P_{n}(\cdot;d\mu_0)$. Consequently, if $\lambda$ is large enough, then $S_n(\cdot;\lambda)$ has $n$ simple,
real zeros that interlace with the zeros of $P_{n-1}(\cdot; d\mu_0)$ and with those of $P_{n-1}(\cdot; d\mu_1)$.

Using \eqref{eq:SOP-OP-A1} and the three-term relation for $P_n(\cdot; d\mu_0)$, it is possible to derive a
recurrence relation for $\{S_n(\cdot;\lambda)\}$, as observed in \cite{BrMe95}. When $\{\Lambda_n \}_{n\geq0}$ is a sequence
of orthogonal polynomials, which holds, by \eqref{eq:limSn}, if $d \mu_1$ is classical, the recurrence relation is
just the three-term recurrence relation of $\{\Lambda_n\}_{n\ge 0}$ written in $S_n(\cdot;\lambda)$ by \eqref{eq:Ldn-Pn}.

\subsection{Determination of coherent pairs}
Because of their applications in Sobolev orthogonal polynomials, an immediate question is to identify all coherent pairs. For this purpose,
the more general notion of orthogonality in terms of linear functionals is often considered. The notion of coherent
pair can be extended to the linear functionals $\{\CU_0, \CU_1\}$, if the relation \eqref{eq:coh-A1} holds with
$P_n(\cdot;d\mu_i)$ replaced by $P_n(\cdot; \CU_i)$. For classical orthogonal polynomials, we allow the
parameters in $\CL^{(\a)}$ and $\CJ^{(\a,\beta)}$ to be negative real numbers but not negative integers.

The first step of identifying all coherent pairs was taken in \cite{MP95}, where the coherent pairs were identified
when either $\CU_0$ or $\CU_1$ is classical in the extended sense. Since the derivative of a classical
orthogonal polynomial is again a classical orthogonal polynomial (with different parameter),  \eqref{eq:coh-A1}
reduces to
$$
  Q_n = P_n + c_n P_{n-1}, \qquad n \ge 1,
$$
where either $\{Q_n\}_{n\geq0}$ or $\{P_n\}_{\geq0}$ is a sequence of classical orthogonal polynomials. Comparing the
coefficients of the three-term relations satisfied by $\{Q_n\}_{n\geq0}$ and $\{P_n\}_{n\geq0}$ in the above identity,
all coherent pairs when one of the measures is classical in the extended sense were found in \cite{MP95}.

The next important step is \cite{MPP95}, which shows that if $\{\CU_0, \CU_1\}$ is a coherent pair, then
\begin{equation} \label{eq:coh-PQ}
   n \frac{P_n(x;\CU_0)} {\|P_n(\cdot; \CU_0)\|_{\CU_0}^2} \CU_0 = \partial (Q_n \CU_1), \quad n \ge 1,
\end{equation}
where, with $a_n$ being the coefficient in \eqref{eq:coh-A1},
$$
  Q_n(x):= a_n \frac{P_n(x;\CU_1)} {\|P_n(\cdot; \CU_1)\|_{\CU_1}^2} -
        \frac{P_{n-1}(x;\CU_1)} {\|P_{n-1}(\cdot; \CU_0)\|_{\CU_0}^2}.
$$

The final step at identifying all coherent pairs was taken in \cite{Mei97}, where the following theorem
was established.

\begin{thm}\label{thm:coh-classical}
If $\{\CU_0,\CU_1\}$ is a coherent pair, then at least one of them has to be classical in the extended sense.
\end{thm}

The proof in \cite{Mei97} started with the observation that the relation \eqref{eq:coh-PQ} with $n=1$ and $n=2$
implies the existence of polynomials $\phi$ of degree at most 3, $\chi$ of degree at most 2 and $\psi$ of degree
exactly 2, such that
\begin{equation}\label{eq:coh-relations}
 \phi \partial \CU_1 = \chi \CU_1, \quad \phi \CU_0 = \psi \CU_1, \quad  \chi \CU_0 = \psi \partial \CU_1,
\end{equation}
where $\phi$, $\chi$ and $\psi$ can be given explicitly in terms of $P_1(\cdot; \CU_0)$, $P_2(\cdot; \CU_0)$, $Q_1$ and $Q_2$.
In particular, $\psi = Q_1 Q_2' -Q_1' Q_2$. The polynomial $\psi$ has two zeros. If the two zeros coincide at the
point $\xi$, then $B'(\xi) = 0$ and, as a result, $Q_1(\xi) =0$, which can be used to show that $\phi(\xi) =0$.
Writing $\phi(x) = (x-\xi) \wt \phi(x)$, it can then be shown, by eliminating $\partial \CU_1$ in \eqref{eq:coh-relations},
that $\wt \phi_0 \CU_0  = c (x-\xi) \CU_1$, where $c$ is a constant, from which it follows that $\CU_0$ satisfies
the Pearson equation
$\partial (\wt \phi \CU_0) = \eta \CU_0$, where $\eta$ is a polynomial of degree $1$. As the solution of the
Pearson equation,  $\CU_0$ has to be classical. When the two zeros of $\psi$ are different, $\CU_1$ can be
shown to be classical; the analysis in this case, however, is more involved.

Together,  \cite{MP95,Mei97} give a complete list of coherent pairs. In the case of $\CU_0$ and $\CU_1$
are positive definite linear functionals associated with measures $d\mu_0$ and $d\mu_1$, these cases are given as follows:

\medskip \noindent
{\bf Laguerre case}
\begin{enumerate}
\item $d\mu_0 (x) = (x-\xi) x^{\a-1} e^{-x}dx$ and $d\mu_1(x) = x^{\a} e^{-x}dx$, where
if $\xi < 0$ then $\a >0$, and if $\xi = 0$ then $\a > -1$.
\item $d\mu_0 (x) = x^{\a} e^{-x}dx$ and $d\mu_1(x) = \frac {x^{\a+1} e^{-x}}{x-\xi}dx + M \delta_{\xi}$, where
if $\xi < 0$, $\a >-1$ and $M \ge 0$.
\item $d\mu_0 (x) =  e^{-x}dx + M \delta_{0}$ and $d\mu_1(x) = e^{-x}dx$, where $M \ge 0$.
\end{enumerate}

\medskip \noindent
{\bf Jacobi case}
\begin{enumerate}
\item $d\mu_0 (x) = |x-\xi| (1-x)^{\a-1} (1+x)^{\beta-1} dx$ and $d\mu_1(x) =(1- x)^{\a}(1+x)^\beta dx$, where
if $|\xi| > 1$ then $\a >0$ and $\beta >0$, if $\xi = 1$ then $\a > -1$ and $\beta > 0$, and if $\xi = -1$ then
$\a > 0$ and $\beta > -1$.
\item $d\mu_0 (x) =(1-x)^{\a} (1+x)^{\beta} dx$ and $d\mu_1(x) = \frac{1}{|x-\xi|}(1- x)^{\a+1}(1+x)^{\beta+1} dx + M \delta_{\xi} $,
where $|\xi| > 1$, $\a > -1$ and $\beta > -1$, and $M \ge 0$.
\item $d\mu_0 (x) =  (1+x)^{\beta-1} dx + M \delta_{1}$ and $d\mu_1(x) = (1+x)^\beta dx$, where $\beta > 0$ and $M \ge 0$.
\item $d\mu_0 (x) =  (1-x)^{\a-1} dx + M \delta_{-1}$ and $d\mu_1(x) = (1-x)^\a dx$, where $\a > 0$ and $M \ge 0$.
\end{enumerate}

The similar analysis was also carried out for symmetrically coherent pairs in the work cited above. They lead to the
following list of symmetrically coherent pairs.

\medskip \noindent
{\bf Hermite case}
\begin{enumerate}
\item $d\mu_0 (x) =  e^{-x^2} dx$ and $d\mu_1(x) = \frac{1}{x^2 + \xi^2} e^{-x^2} dx$, where $\xi \ne 0$.
\item $d\mu_0 (x) = (x^2 + \xi^2) e^{-x^2} dx$ and $d\mu_1(x) = e^{-x^2} dx$, where $\xi \ne 0$.
\end{enumerate}

\medskip \noindent
{\bf Gegenbauer case}
\begin{enumerate}
\item $d\mu_0 (x) = (1-x^2)^{\a-1} dx$ and $d\mu_1(x) = \frac{1}{x^2 + \xi^2} (1-x^2)^\a dx$,
  where $\xi \ne 0$ and $\a >0$.
\item $d\mu_0 (x) = (1-x^2)^{\a-1} dx$ and $d\mu_1(x) = \frac{1}{\xi^2 - x^2} (1-x^2)^\a dx + M \delta_{\xi}
   + M \delta_{-\xi}$,
  where $|\xi | \ge 1$, $\a >0$ and $M \ge 0$.
\item $d\mu_0 (x) = (x^2+\xi^2)(1-x^2)^{\a-1} dx$ and $d\mu_1(x) = (1-x^2)^\a dx$, where $\a  >0$.
\item $d\mu_0 (x) = (\xi^2-x^2)(1-x^2)^{\a-1} dx$ and $d\mu_1(x) = (1-x^2)^\a dx$, where $|\xi| \ge 1$ and
$\a  >0$.
\item $d\mu_0(x) = dx + M \delta_{1} + M \delta_{-1}$ and $d\mu_1(x) = dx$, where $M \ge 0$.
\end{enumerate}

\subsection{Generalized coherent pairs}
We deduced the identity \eqref{eq:SOP-OP-A1} from the definition  \eqref{eq:coh-A1} of the coherent pair. In the
reverse direction, however, \eqref{eq:coh-A1} does not follow from the identity \eqref{eq:SOP-OP-A1}, as
observed in \cite{KKMY02}. We restate  \eqref{eq:SOP-OP-A1} below,
\begin{equation*}
  S_n(x;\lambda) + b_{n-1}(\lambda) S_{n-1}(x;\lambda) =  P_n(x;d\mu_0) + \wh a_{n-1} P_{n-1}(x;d\mu_0), \quad n \ge 1.
  \tag{\ref{eq:SOP-OP-A1}'}
\end{equation*}
For convenience, we let $\sS_n(x)$ denote the left hand side of (\ref{eq:SOP-OP-A1}'). Clearly $\sS_n'$ can be
expanded in terms of $\{P_k(\cdot; d\mu_1)\}$,
$$
    \sS_n'(x) = n P_{n-1}(x;d\mu_1) + \sum_{k=0}^{n-2} d_{k,n} P_k(x;d\mu_1), \qquad d_{k,n} =
         \frac{\la \sS_n', P_k(\cdot; d\mu_1)\ra_{d \mu_1}} {\|P_k(\cdot;d\mu_1)\|_{d\mu_1}^2}.
$$
For $0 \le j \le n-2$, it follows directly from the definition of $S_n$ that $\la \sS_n, P_j(\cdot;d\mu_1) \ra_\lambda =0$
and it follows from \eqref{eq:SOP-OP-A1} that $\la \sS_n, P_j(\cdot;d\mu_1) \ra_{d\mu_0} =0$. Consequently,
by the definition of $\la \cdot; \cdot \ra_\lambda$, we must have $\la \sS_n',  P_j(\cdot;d\mu_1) \ra_{d\mu_1} =0$ for
$0 \le j \le n-2$, which implies that $d_{k,n} = 0$ if $0\le k \le n-2$. Hence,
$$
   \sS_n'(x) = P_{n}' (x;d \mu_0) + \wh a_{n-1} P_{n-1}'(x;d\mu_0) = n P_{n-1}(x;d\mu_1)  + d_{n-2,n} P_{n-2}(x;d\mu_1).
$$
Recall that $\wh a_n = (n+1) a_n /n$. Setting $b_{n-2} = d_{n-2,n}/n$ and shift the index from $n$ to $n+1$, we
conclude the following relation between $\{P_n(\cdot;d\mu_0)\}$ and $\{P_n(\cdot;d\mu_1)\}$,
\begin{equation}\label{eq:cohAA3}
 P_{n}(x;d\mu_1)  + b_{n-1} P_{n-1}(x;d\mu_1) =  \frac{P_{n+1}' (x;d \mu_0)}{n+1} + a_{n} \frac{P_{n}'(x;d\mu_0)}{n},
   \quad n \ge 1.
\end{equation}
Thus, in the reverse direction, \eqref{eq:SOP-OP-A1} leads to \eqref{eq:cohAA3} instead of \eqref{eq:coh-A1}.

Evidently, \eqref{eq:cohAA3} is a more general relation than that of \eqref{eq:coh-A1}. This suggests the
following definition.

\begin{defn}
The pair $\{d\mu_0, d\mu_1\}$ is called a generalized coherent pair if \eqref{eq:cohAA3} holds for all $n \ge 1$, and
this definition extends to the linear functionals $\{\CU_0, \CU_1\}$.
\end{defn}

\begin{thm}
Let $\CU_0, \CU_1$ be two linear functionals. Then $\{\CU_0,\CU_1\}$ is a generalized coherent pair if and
only if  \eqref{eq:SOP-OP-A1} holds.
\end{thm}

\begin{proof}
We have already established that \eqref{eq:SOP-OP-A1} implies \eqref{eq:cohAA3}. Assume now that
\eqref{eq:cohAA3} holds. Let $\sT_{n+1}(x): = P_{n+1}(x;\CU_0) + \wh a_n P_n(x; \CU_0)$, where
$\wh a_n = (n+1) a_n/n$. For any polynomial  $p_{n-1}$ of degree at most $n-1$, it follows directly from the
definition that $\la \sT_{n+1}, p_{n-1} \ra_{\CU_0} = 0$ and it follows from \eqref{eq:cohAA3} that
$\la \sT_{n+1}', p_{n-1}' \ra_{\CU_1} =0$. Consequently, $\la \sT_{n+1}, p_{n-1}\ra_\lambda =0$. Hence, expanding
$\sT_{n+1}$ in terms of $S_n(x;\lambda)$, we see that
$$
   \sT_{n+1}(x) = S_{n+1}(x;\lambda) + b_n(\lambda) S_n(x;\lambda), \qquad n \ge 0,
$$
which is precisely \eqref{eq:SOP-OP-A1} with $n$ replaced by $n+1$.
\end{proof}

Two examples of generalized coherent pairs were given in \cite{BR01, BBR03} when $\CU_0$ is a polynomial
perturbation of degree $1$ of the Laguerre or Jacobi linear functional.

\begin{exam}
In these examples, $\{d\mu_0, d\mu_1\}$ is not a coherent pair but a generalized coherent pair. For the
first example, $\a, \beta > -1$ and $|\xi_0|, |\xi_1| \ge 1$,
$$
 d\mu_0 = (1-x)^\a (1+x)^\beta dx, \quad  d\mu_1 = \frac{x-\xi_0}{x-\xi_1}(1-x)^{\a+1} (1+x)^{\beta +1}dx + M\delta_{\xi_1}.
$$
For the second example, $\a > -1$ and $\xi \le 0$,
$$
d\mu_0 = x^\a e^{-x}dx \quad d\mu_1 =  \frac{x-\xi_0}{x-\xi} x^{\a+1} e^{-x} dx + M\delta_{\xi}.
$$
\end{exam}

As in the case of coherent pairs, one can ask the question of identifying all generalized coherent pairs.
If $\CU_0$ is a classical linear functional, then the answer to this question was given in \cite{AMPR03}.
The complete identification was achieved in \cite{DM04}, which states that either $\CU_0$ or $\CU_1$ must
be a semiclassical linear functional.

\begin{defn}
Let $s$ be a nonnegative integer. A linear functional $\CU$ is called semiclassical of class $s$ if there exist
polynomials $\phi$ and $\psi$ such that $s= \max \{\deg \phi -2, \deg \psi -1\}$ and
$\partial (\phi \CU) = \psi \CU$.
\end{defn}

An in-depth study of orthogonal polynomials with respect to a semiclassical linear functional is given in \cite{Mar91},
which contains the following theorem.

\begin{thm} \label{thm:semiclassical}
For a semiclassical linear functional $\CU$ of class $s$, the following are equivalent:
\begin{enumerate}[\quad 1.]
\item $\{P_n\}_{n\geq0}$ is a sequence of monic orthogonal polynomials with respect to $\CU$;
\item There exists a monic polynomial $\phi$ such that $\la \phi \CU, P_n' Q\ra =0$ for all $Q \in \Pi_{n-s-1}$.
\item There exists a monic polynomial $\phi$ of degree $r$ and constants $a_{n,j}$ such that
$$
  \phi(x) P_{n+1}'(x)= \sum_{j=n-s}^{n+r} a_{n,j} P_j(x), \quad \hbox{where} \quad a_{n,n-s} \ne 0.
$$
\end{enumerate}
\end{thm}

Every semiclassical linear functional of class $0$ is classical. The next case, semiclassical of
class $1$, is used to characterize the generalized coherent pairs. The main result in \cite{DM04} is the following
theorem.

\begin{thm}
If $\{\CU_0, \CU_1\}$ is a generalized coherent pair, then at least one of them has to be semiclassical of
class at most $1$.
\end{thm}

The proof is based on the observation that if $\{\CU_0, \CU_1\}$ is a generalized coherent pair, then
there are polynomials $Q_n$ of degree $n$ and $R_{n+1}$ of degree at most $n+1$, such that
$$
       \partial (Q_n \CU_1) =  R_{n+1} \CU_0, \qquad n \ge 1,
$$
which is an analogue of \eqref{eq:coh-PQ}. This relation with $n =1$ and $n =2$ is used to deduce the
existence of polynomials $\phi$ of degree at most 4, $\chi$ of degree at most 3 and $\psi$ of degree
exactly 2, such that
\begin{equation} \label{eq:coh-relations2}
 \phi \partial \CU_1 = \chi \CU_1, \quad \phi \CU_0 = \psi \CU_1, \quad  \chi \CU_0 = \psi \partial \CU_1,
\end{equation}
where $\phi$, $\chi$ and $\psi$ can be given explicitly in terms of $R_1$, $R_2$, $Q_1$ and $Q_2$.
In particular, $\psi = R_1 R_2' - R_2$. This is an analogue of \eqref{eq:coh-relations} for the coherent pair,
in which the degrees of $\phi$ and $\chi$ are increased by 1. The polynomial $\psi$ has two zeros. If
the two zeros coincide, say at $\xi$, then it can be shown that $\phi(x) = (x-\xi) \wt \phi(x)$ and $\CU_0$
satisfies $\partial (\wt \phi \CU_0) = \eta \CU_0$, where $\eta$ is a polynomial of degree at most 2. This
shows, by Theorem \ref{thm:semiclassical}, that $\CU_0$ is semiclassical of class at most 1. In the
case of that $\psi$ has two distinct zeros, a more involved analysis shows that $\phi$ vanishes at
one of the zeros, say $\xi$, so that $\phi(x) = (x-\xi)\wt \phi(x)$, which can be used to show that
$\CU_1$ has to be semiclassical of class at most 1. Furthermore, in the latter case, $\CU_0$ is
determined by $\wt \phi \CU_0 = (x-\xi) \CU_1$, where $\xi$ is either one of the two zeros of $\psi$.

The complete identification of generalized coherent pairs can be worked out by examine all possible
$\wt \phi$. Up to a linear change of variables, if $\deg \wt \phi = 3$, it has three canonical
forms, $\wt \phi(x) = x^3$, $x^2(x-1)$, or $x(x-1)(x-\lambda)$; if $\deg \wt \phi = 2$, it has two
canonical forms, $\wt \phi(x) = x^2$ or $x(x-1)$; if $\deg \wt \phi =1$, then $\wt \phi(x) =x$. The
arduous work of analyzing all cases was carried out meticulously in \cite{DM04}, where the complete
list of generalized coherent pairs can be found in two long tables.

\subsection{$(M,N)$ coherent pairs} The generalized coherent pair suggests immediately further
generalization of coherent pair by extending \eqref{eq:cohAA3}.

\begin{defn}
Let $M, N$ and $m, n$ be nonnegative integers. A pair of linear functional $\{\CU_0, \CU_1\}$
is said to be a $(M,N)$-coherent pair of order $(m, n)$ if, for $k \ge 0$,
$$
 \frac{P_{k+m}^{(m)}(x, \CU_0)}{(k+1)_m} + \sum_{i=1}^M a_{i,k} \frac{P_{k-i+m}^{(m)}(x,\CU_0)}{(k-i+1)_m}
     = \frac{P_{k+n}^{(n)}(x, \CU_1)}{(k+1)_n} + \sum_{i=1}^N b_{i,k} \frac{P_{k-i+n}^{(n)}(x,\CU_1)}{(k-i+1)_n},
$$
where $a_{i,k}$ and $b_{i,k}$ are complex numbers such that $a_{M,k} \ne 0$ if $k \ge M$,
$b_{N,k} \ne 0$ if $k \ge N$, and $a_{i,k} = b_{i,k} = 0$ if $i > k$.
\end{defn}

In this definition, the coherent pair is the $(0,1)$-coherent pair of order $(0,1)$. The generalized coherent pair
is $(1,1)$-coherent pair of order $(0,1)$. The $(0,2)$-coherent pair of order $(0,1)$ was considered in \cite{KLM01}
and the $(0,N)$-coherent pair of oder $(0,1)$ was called $N$-coherent pair and studied in \cite{MMM01}.
The general setting of $(M,N)$-coherent pairs of order $(m,n)$ was studied more recently in \cite{JP08, JP13, JMPP14}.

Let $a_{0,k} :=1$ for $0 \le k \le N-1$ and $b_{0,k}: =1$ for $0 \le k \le M-1$. Set
\begin{align*}
A_{N}:= & \left[\begin{matrix}
    a_{0,0} & \cdots   & \cdots & \!\!  a_{M,M} & & \bigcirc  \\
      & a_{0,1} & \cdots & \cdots & \!\!  \!\! a_{M,M+1} &  \\
     & &  \ddots & \cdots &   \quad \ddots & \\
     \bigcirc &&&  a_{0,N-1} & \ldots & \!\! \!\!  a_{M, M+N-1}
  \end{matrix}\right], \\
B_M := & \left[\begin{matrix}
    b_{0,0} & \cdots   & \cdots & \!\!  b_{N,N} & & \bigcirc  \\
      & b_{0,1} & \cdots & \cdots & \!\!  \!\! b_{N,N+1} &  \\
     & &  \ddots & \cdots &   \quad \ddots & \\
     \bigcirc &&&  b_{0,M-1} & \ldots & \!\! \!\!  b_{N, M+N-1}
  \end{matrix}\right], \qquad   L_{M+N}:= \left[ \begin{matrix} A_N \\ B_M \end{matrix} \right].
\end{align*}
Then $A_N$ is a matrix of size $N \times (M+N)$, $B_M$ is a matrix of size $M \times (M+N)$, and
$L_{M+N}$ is a square matrix of size $(M+N ) \times (M+N)$.

\begin{thm}
Let $\{\CU_0, \CU_1\}$ be a $(M,N)$-coherent pair of order $(m, n)$ with $m \ge n$ and assume that
$L_{M+N}$ is a non-singular matrix. Then there exist polynomials $Q_{M+n+k}$ of degree $M+n+k$ and
$R_{N+m+k}$ of degree at most $N+m+k$, such that
\begin{equation} \label{eq:gen-coh-QR}
       \partial^{m-n} (Q_{M+n+k} \CU_1) =  R_{N+m+k} \CU_0, \qquad k \ge 0.
\end{equation}
Furthermore, the following classification holds
\begin{enumerate}[\quad \rm (1)]
 \item If $m =n$, then $\CU_0$ is semiclassical if and only if $\CU_1$ is semiclassical.
 \item If $m > n$, then $\CU_0$ and $\CU_1$ are both semiclassical.
\end{enumerate}
\end{thm}

The equation \eqref{eq:gen-coh-QR} is an extension of \eqref{eq:coh-PQ}. This theorem was established in
\cite{JMPP14} and it extends the results in \cite{JP08, JP13,MP12}.

The connection between these extended coherent pairs and the Sobolev orthogonal polynomials is illustrated
in the case of $(M,N)$-coherent pair of order $(m,0)$, which is called $(M,N)$-coherent pair of order $m$.
Assume that $\CU_0$ and $\CU_1$ are positive definite and represented by $d\mu_0$ and $d\mu_1$,
respectively. Fix $m \in \NN$ and define the Sobolev inner product by
$$
  \la f, g \ra_{\lambda} := \int_{\RR} f(x) g(x) d\mu_0 + \lambda \int_{\RR} f^{(m)} (x)  g^{(m)} (x) d\mu_1, \quad \lambda > 0.
$$
Let $\{S_n(\cdot; \lambda)\}$ denote the corresponding sequence of monic orthogonal polynomials.

\begin{thm}
Let $\{\mu_0,\mu_1\}$ be a $(M,N)$-coherent pair of order $m$. If $n < m,$ then $S_n(x;\lambda) = P_n(x; d\mu_1)$
and, for $n\geq0,$
$$
   \frac{P_{n+m}(x;d\mu_0)}{(n+1)_m} + \sum_{i=1}^M a_{i,n} \frac{P_{n-i+m}(x;d\mu_0)}{(n-i+1)_m}
    =  \frac{S_{n+m}(x;\lambda)}{(n+1)_m}+ \sum_{j=n+1}^{\max\{M,N\}} c_{j,n}(\lambda) \frac{S_{n-j+m}(x; \lambda)}{(n-j+1)_m},
$$
where $c_{j,n}(\lambda)$ are constants that can be determined recursively.
\end{thm}

This theorem was also established in \cite{JMPP14}, where an algorithm is given for computing $c_{j,n} (\lambda)$ and
the Fourier coefficients of the orthogonal expansion in $S_n(\cdot;\lambda)$.

As in the case of coherent or generalized coherent pairs, one can consider the problem of identifying all
coherent pairs.  The problem, however, is open in all cases beyond the two cases that we have discussed.
At the time of this writing, the $(M,N)$-coherent pairs of order $(m,n)$ is still under study.

\section{Classical orthogonal polynomials as Sobolev orthogonal polynomials}
\setcounter{equation}{0}

The classical orthogonal polynomials are Sobolev orthogonal polynomials themselves, since derivatives of a
classical orthogonal polynomial are classical orthogonal polynomials of the same type. Indeed, for the Hermite
polynomials $H_n$, $H_n'(x) = 2 n H_{n-1}(x)$, so that $\{H_n\}_{n\ge 0}$ is also the sequence of orthogonal
polynomials with respect to the Sobolev inner product
\begin{equation}\label{HermiteSOPm}
 \la f,g\ra_\CH := \sum_{k=0}^m \int_{\RR} f^{(k)}(x)g^{(k)}(x) e^{-x^2} dx, \qquad m =1,2,\ldots,
\end{equation}
where $\lambda_0 > 0$ and $\lambda_k \ge 0$ for $k =1,\ldots, m$, which we assume for the two cases below as well.
For the Laguerre polynomials $L_n^{(\a)}$, $\frac{d}{dx} L_n^{(\a)}(x) = - L_{n-1}^{(\a+1)}(x)$, so that
$\{L_n^{(\a)}\}_{n\ge 0}$ is also the sequence of orthogonal polynomials with respect to the Sobolev
inner product
\begin{equation}\label{LaguerreSOPm}
 \la f,g\ra_{\CL^\a} := \sum_{k=0}^m  \lambda_k  b^{(\a+k)}  \int_0^\infty f^{(k)}(x)g^{(k)}(x)  x^{\a+k} e^{-x} dx, \quad m = 1,2,\ldots,
\end{equation}
where $b^{(\a)} :=1/\Gamma(\a+1)$ and $\a > -1$. For the Jacobi polynomials $P_n^{(\a,\beta)}$, it is known that
$\frac{d}{dx} P_n^{(\a,\beta)}(x) =
\frac{n+\a+\beta+1}{2}P_{n-1}^{(\a+1,\beta+1)}(x)$, so that $\{P_n^{(\a,\beta)}\}_{n\ge 0}$ is also the sequence of the orthogonal
polynomials with respect to the Sobolev inner product
\begin{equation}\label{JacobiSOPm}
 \la f,g\ra_{\CJ^{\a,\beta}}:= \sum_{k=0}^m  \lambda_k b^{(\a+k,\beta+k)}  \int_{-1}^1 f^{(k)}(x)g^{(k)}(x) (1-x)^{\a+k}(1+x)^{\beta+k} dx,
\end{equation}
where $m=1,2,\ldots$ and $b^{(\a,\beta)} := \Gamma(\a+\beta+2)/(2^{\a+b+1} \Gamma(\a+1)\Gamma(\beta+1))$, and $\a,\beta > -1$.

For arbitrary $\a \in \RR$, the Laguerre polynomials can be defined by the explicit formula
(\cite[(5.1.6)]{Szego})
$$
   L_n^{(\a)} (x) =  \sum_{\nu =0}^n  \frac{(\a + \nu +1)_{n-\nu} }{(n-\nu)!} \frac{(-x)^\nu}{\nu!}, \qquad n \ge 0,
$$
which turns out to be orthogonal with respect to a Sobolev inner product. For $k \in \NN$, define a
matrix $M(k)$ by
$$
   M(k) = \left[ m_{i,j}(k) \right ]_{i,j=0}^k := Q(k) Q(k)^T, \qquad Q(k) := \left[ (-1)^{j-i} \binom{k-i}{j-i} \right]_{i,j=0}^k,
$$
where we assume that $\binom{a}{b} =0$ if $b < 0$. Then $M(k)$ is positive definite.

\begin{thm} \label{LagSOPgeneral}
For $\a \in \RR$, let $k := \max\{0, \lfloor \a \rfloor\}$. Then the sequence of Laguerre polynomials
$\{L_n^\a\}_{n \ge 0}$ is orthogonal with respect to the inner product
$$
  \la f, g \ra_{\CL^\a} := \int_{0}^\infty \left(f(x), f'(x), \ldots, f^{(k)}(x)\right) M(k) \left(g(x), g'(x), \ldots, g^{(k)}(x)\right)^T x^{\a+k} e^{-x} dx.
$$
Moreover, if $- \a \in \NN$, then $k = - \a$ and one can write, by integration by parts,
\begin{align*}
  \la f, g\ra_{\CL^{-k}}   = & \frac 12 \sum_{i=0}^{k-1} \sum_{j=0}^i m_{i,j}(k) \left [ f^{(i)}(0)g^{(j)}(0) +
         f^{(j)}(0)g^{(i)}(0) \right] + \int_0^\infty f^{(k)}(x) g^{(k)}(x) e^{-x} dx.
\end{align*}
\end{thm}

This theorem was proved in \cite{PP96}. The case $- a \in \NN$ was proved in \cite{KL95}. In the case of
$-a \in \NN$, the result can be deduced directly with the help of the relation
\begin{equation}\label{eq:Laguerre-k}
   L_n^{(-k)}(x) =  (-x)^k \frac{(n-k)!}{n!} L_{n-k}^{(k)}(x), \quad n \ge k,
\end{equation}
which shows, in particular, that $L_n^{(-k)}$ and its up to $(k-1)$-th derivatives vanish at $x =0$, so that the
orthogonality follows immediately from $\frac{d}{d x} L_n^{(\a)}(x) = - L_{n-1}^{(\a+1)}(x)$.

In the Jacobi case, it is known that if $-\a, - \beta, - \a-\beta \in \RR \setminus \NN$, then the Jacobi polynomials are
orthogonal with respect to some Sobolev inner product of the form
$$
      \la f, g\ra = \int_{\RR} f(x) g(x) d\mu_0(x) + \int_{\RR} f'(x) g'(x) d\mu_1(x),
$$
where $\mu_0$ and $\mu_1$ are real, possibly signed, Borel measure on $\RR$. In the case of $\a = - k$,
$k \in \NN$, and $- \beta -k \in \RR \setminus \NN$, more can be said since
\begin{equation}\label{eq:Jacobi-k}
  \binom{n}{k} P_n^{(-k,\beta)}(x) = \binom{n+\beta}{k} \left( \frac{x-1}{2} \right)^k P_{n-k}^{(k,\beta)}(x),
\end{equation}
and an analogous expression for $P_n^{(\a,-k)}$ follows from $P_n^{(\a,\beta)}(x) = (-1)^n P_n^{(\beta,\a)}(-x)$.
It follows that the $P_n^{(-k,\beta)}$ and its up to $(k-1)$th derivatives vanish at $x =1$. Furthermore,
$\frac{d^k}{dx^k} P_n^{(-k,\beta)}(x) = c P_{n-k}^{(0,\beta+k)}(x)$. These relations can be used to derive a
Sobolev orthogonality of the Jacobi polynomials. Let $\Lambda(k) = \mathrm{diag} \{\lambda_1,\ldots,\lambda_k\}$ with $\lambda_j > 0$.
Define
$$
   M(k) = Q(k)^{-1} \Lambda(k) (Q(k)^{-1})^T, \qquad Q(k) = \left[ \left(\partial^j P_i^{(-k,\beta)}\right)(1)\right]_{i,j=0}^{k-1}
$$

\begin{thm}
For $k \in \NN$ and $\beta \in \RR$ such that $ -(k+ \beta) \notin \NN$,  the sequence of Jacobi polynomials $\{P_n^{(-k,\beta)}\}_{n \ge 0}$ is orthogonal with
respect to the inner product
\begin{align*}
  \la f, g\ra_{\CJ^{-k,\beta}} := & \int_{-1}^1 f^{(k)}(x) g^{(k)}(x) (1+x)^{k+\beta} dx \\
     &+ \left(f(1), f'(1), \ldots, f^{(k-1)}(1)\right) M(k) \left(g(1), g'(1), \ldots, g^{(k)}(1)\right)^T.
\end{align*}
\end{thm}

An analogue result can be stated for $\a \in \RR$ and $-(\a+k) \notin \NN$, where the point evaluations need
to be at $-1$. This theorem was proved in \cite{ARPP}, which contains an analogue result for the Laguerre
polynomials $L_n^{(-k)}$ that is more general than the corresponding result in Theorem \ref{LagSOPgeneral}.
These results extend several special cases appeared early in the literature (see, for example, \cite{KL98}).

The similar results also hold for the Jacobi polynomials with $\a = -k$, $\beta=-l$ with $k,l \in \NN$, where the
orthogonality is defined with respect to the bilinear form
$$
   \la f, g\ra = \int_{-1}^1 f^{(k+l)}(x) g^{(k+l)}(x) (1-x)^l (1+x)^{k} dx + \la f, g \ra_D,
$$
where $\la f, g\ra_D$ is defined via point evaluations of $f$ and $g$ and their derivatives up to $(k-1)$-th
order at $1$ and up to $(l-1)$-th order at $-1$ (\cite{AAR02, APP98}). In this case, the Jacobi polynomial
$P_n^{(-k,-l)}$ vanishes identically for $\max \{k,l\} \le n < k+l$ and has reduced degree if $(k+l)/2 \le n < \max\{k,l\}$,
so that it is necessary to define the Jacobi polynomials with such indexes. There are several ways to redefine these
polynomials. Let $\wh P_n^{(\a,\beta)}$ denote the usual monic Jacobi polynomial of degree $n$. In \cite{AAR02},
the monic Jacobi polynomials $\wh P_n^{(\-k,-l)}$ are defined by
$$
  \wh P_n^{(\-k,-l)}(x) = \frac12 \left[ \lim_{\a \to -k} P_m^{(\a, -l)}(x)+\lim_{\beta \to -l} P_m^{(-k, \beta)}(x)\right],
$$
where $m = k + l -n -1$ if $(k+l)/2 \le n < \max\{k,l\}$ and $m =n$ otherwise.

It is worth to mention that the Jacobi polynomials with negative indexes have been used extensively in the spectral
theory for solving differential equations. See, for example, \cite{GSW09, SWL, LX13} and the references therein.
In the spectral theory, orthogonal expansions and approximation by polynomials in Sobolev spaces are used, so that
Sobolev orthogonal polynomials are needed, although they often appear implicitly.

\section{Sobolev orthogonal polynomials of the second type}
\setcounter{equation}{0}

As stated in the introduction, an inner product is called a Sobolev inner product of the second type if the derivatives
appear only on function evaluations on a finite discrete set. More precisely, such an inner product takes the form
\begin{equation}\label{SobolevtypeOPm}
 \la f,g\ra_S :=  \int_{\RR} f(x) g(x) d\mu_0 + \sum_{k=1}^m \int_{\RR} f^{(k)}(x)g^{(k)}(x) d\mu_{k},
\end{equation}
where $d\mu_{0}$ is a positive Boreal measure supported on an infinite subset of the real line and $d\mu_{k}$,
$k=1, 2, ..., m$, are positive Borel measures supported on finite subsets of the real line. Let $\delta_c$ denote
the delta measure supported at the point $c \in \RR$, that is, $\delta_c f(x) = f(c)$. In most cases considered
below, $d\mu_k = A_k \delta_c$ or $d\mu_k = A_k \delta_a + B_k \delta_b$, where $A_k$ and $B_k$ are nonnegative
numbers.

Orthogonal polynomials for such an inner product are called Sobolev orthogonal polynomials of discrete type. The
first study was carried out for the classical weight functions. The Laguerre case was studied in \cite{Koek90}
with $d\mu_{0}= x^{\alpha} e^{-x} dx$, $\a > -1$, and $d\mu_{k}= M_{k} \delta_{0}, k=1, 2, ..., m$; the $n$-th
Sobolev orthogonal polynomial, $S_n$, is given by
$$
     S_n(x) = \sum_{k=0}^{\min\{n, m+1\}} (-1)^k A_k L_{n-k}^{\a+k}(x),
$$
in which $A_k$ are constants determined by a linear system of equations. The Gegenbauer case was
studied in \cite{BavMei89,BavMei} with $d\mu_{0}= (1- x^{2})^{\lambda- 1/2} dx + A (\delta_{-1} + \delta_{1})$, $\lambda > -1/2$,
and $m=1$, $d\mu_{1}= B ( \delta_{-1} + \delta_{1})$;  the $n$-th Sobolev orthogonal polynomial is given
by
$$
  S_n(x) =  \sum_{k=0}^{2} a_{k,n} x^{k} C_{n-k}^{\lambda+k}(x)
$$
where $a_{0,n}$, $a_{1,n}$, $a_{2,n}$ are appropriate constants. In both cases, the Sobolev orthogonal polynomials satisfy
higher (than three) order recurrence relations that expands $q(x) S_n(x)$ as a sum of $S_m$. For the Laguerre
case, $q(x) = x^{m+1}$ and the order is  $2m+3$; for the Gegenbauer case, either $q(x) = (1-x^2)^2$ and the order is 9
or $q(x) = x^3-3x$ and the order is 7. Furthermore, in both cases, the polynomial $S_n$ satisfies a second order linear differential equation with polynomial coefficients, whose degree does not depend on $n$.

When $M_{k}=0$,  $k=1, 2, ..., m-1$, and $d\mu_{m}= M_{m} \delta_{c}$, the inner product \eqref{SobolevtypeOPm}
becomes
\begin{equation}\label{SobolevtypeOPm2}
 \la f,g\ra_m :=  \int_{\RR} f(x) g(x) d\mu_0 +  M_m f^{(m)}(c)g^{(m)}(c),
\end{equation}
where $c\in\RR$ and $M_m \ge 0$. This case was studied in \cite{MarRo}. Let $\{P_{n}\}_{n\geq0}$ and
$\{S_{n}\}_{n\geq0}$ denote the sequences of monic orthogonal polynomials with respect to $d\mu_{0}$ and
$\la \cdot, \cdot \ra_m$, respectively. For $i,j \in \NN_0$, define
$$
K_{n-1}^{(i,j)}(x,y):=  \sum_{l=0}^{n-1} \frac {P_{l}^{(i)}(x) P_{l}^{(j)}(y)}{||P_{l}||_{d\mu_{0}}^{2}}.
$$
Notice that $K_{n-1}^{(0,0)}$ is the reproducing kernel of the $(n-1)$-th partial sum of orthogonal expansion with
respect to $d\mu_0$. It was shown in \cite{MarRo} that
\begin{equation}\label{SobolevtypeOPeq}
 S_{n}(x) = P_{n} (x) - \frac { M_{m} P^{(m)}_{n} (c)}{ 1 + M_{m} K_{n-1} ^{(m,m)}(c,c)} K_{n-1}^{(0,m)}(x,c),
\end{equation}
which extends the expression for $m=0$ in \cite{Krall}. From this relation, one deduces immediately that
$$
  S_{n+1}(x)+ a_{n} S_{n}(x) = P_{n+1} (x) + b_{n}P_{n} (x), \qquad n \ge 0,
$$
where $a_n$ and $b_n$ are constants that can be easily determined. This shows a structure  similar to
\eqref{eq:SOP-OP-A1} that one derives for the Sobolev orthogonal polynomials in the case of coherent pair.
The Sobolev polynomials $S_n$ also satisfy a higher order recurrence relation
\begin{equation}\label{SobolevtypeRR}
      (x-c)^{m+1} S_{n}(x) = \sum_{j= n- m-1}^{n+m+1} c_{n,j} S_{j}(x),
\end{equation}
where $c_{n, n+m+1}=1$ and $c_{n, n-m-1} \neq 0$.

Orthogonality of a sequence of polynomials is determined by the three term recurrence relations that it satisfies,
the precise statement of which is known as the Favard theorem. One may ask if there is a similar theorem for
polynomials satisfying higher order recurrence relations. There are two types of results in this direction, both
related to Sobolev orthogonal polynomials.

The first one gives a characterization of an inner product $\la \cdot, \cdot \ra$ for which orthogonal polynomials
satisfy the recurrence relation of the form \eqref{SobolevtypeRR}, which holds if the operator of multiplication by
$M_{m,c}: = (\cdot - c)^{m+1}$ is symmetric, {\it i.e.}, $\la M_{m,c} p, q\ra = \la p, M_{m,c} q\ra$.
It was proved in \cite{Du} that if $\la \cdot, \cdot \ra$ is an inner product such that $M_{m,c}$ is symmetric and
it commutes with the operator $M_{0,c}$, {\it i.e.}, $\la M_{m,c} p, M_{0,c} q\ra = \la M_{0,c} p, M_{m,c} q\ra$,
then there exists a nontrivial positive Borel measure $d\mu_{0}$ and a real, positive semi-definite matrix $A$ of
size $m+1$, such that the inner product is of the form
\begin{align}\label{Sobolevtypeform}
\la p, q \ra = & \int_{\RR} p(x)q (x) d\mu_{0}  \\
    & + \left(p(c), p'(c), \ldots, p^{(m)}(c)\right) A \left(q(c), q'(c), \ldots, q^{(m)}(c)\right)^T. \notag
\end{align}
Necessary and sufficient conditions for $A$ being diagonal were also given in \cite{Du}. Furthermore,
a connection between such Sobolev orthogonal polynomials and matrix orthogonal polynomials was established
in \cite{DuVan}, by representing the higher order recurrence relation as a three term recurrence relation with
matrix coefficients for a family of matrix orthogonal polynomials defined in terms of the Sobolev
orthogonal polynomials.

The second type of Favard type theorem was given in \cite{EvLiMarMarkRon}, where it was proved that
the operator of multiplication by a polynomial $h$ is symmetric with respect to the inner product
\eqref{SobolevtypeOPm} if and only if $d\mu_{k}, k=1, 2,..,m$, are discrete measures whose supports are
related to the zeros of $h$ and their derivatives. Consequently, higher order recurrence relations for
Sobolev inner products appear only in Sobolev inner product of the second type.

For the inner product \eqref{Sobolevtypeform}, a sequence of Sobolev orthogonal polynomials can be constructed
using an extension of the method used to obtain \eqref{SobolevtypeOPm}. A more useful approach is to use
the monic orthogonal polynomials, $\wh{P}_{n}$, with respect to the measure $d \mu^*_{0} = (x-c)^{p} d\mu_{0}$,
where $p = 2\lfloor m/2\rfloor +2$, which leads to  (\cite{EvLiMarMarkRon})
\begin{equation}\label{SobolevtypeOPcon}
 S_{n}(x) =\sum_ {k=0}^{p} d_{n,k} \wh{P}_{n-k} (x), \qquad n\geq p,
\end{equation}
where $d_{n,0}= 1$ and $d_{n,p} \neq 0$.
The above relation means that the sequence $\{S_{n}\}_{n\geq0}$ is quasi-orthogonal with respect to
the positive Borel measure $d{\mu}^*_{0}$, a fact that plays a central role in the analysis of zeros and asymptotics
of these polynomials. If the point $c$ is located outside the support of the measure $d\mu_{0}$, it is enough to consider
$d{\mu}^*_{0}= (x-c)^{m+1} d\mu_{0}$, in which case,  \eqref{SobolevtypeOPcon} is satisfied with $m+1$ instead of $p$ and
the formula has the minimal length. If $c$ belongs to the support of the measure $d\mu_{0}$ and $m$ is an even
integer, then $(x-c)^{m+1} d\mu_{0}$ becomes a signed measure, which does not warrant the existence of orthogonal polynomials. This is the reason behind the definition of $p$ in \eqref{SobolevtypeOPcon}. Let $H$ be the banded matrix
that defines the higher order recurrence relation for $\{S_n\}_{n\ge 0}$. Then $H$ can be derived from the Jacobi matrix
$\wh{J}$ associated with the sequence $\{\wh{P}_{n}\}_{n\geq0}$ of monic orthogonal polynomials. Indeed, consider
the decomposition $(\wh J-cI)^{p}=UL$, where $U$ is an upper triangular matrix and $L$ is a lower triangular matrix
with 1 as its diagonal entries, then $H= LU$ (\cite{DerMar}). This gives a direct connection between Sobolev orthogonal
polynomials of discrete type and the iteration of the canonical Geronimus transformation of the measure
$d{\mu}^* = (\cdot - c) d\mu$.

\section{Differential equations}
\setcounter{equation}{0}

Classical orthogonal polynomials are solutions of a second order linear differential equation of the form
\begin{equation} \label{eq:DEQ}
    \a_2(x) y'' + \a_1(x) y' = - \lambda_n y, \qquad n = 0,1,\ldots,
\end{equation}
where $\a_2$ and $\a_1$ are real-valued functions and $\lambda_n$ are real numbers, called eigenvalues. The equation
is called admissible if $\lambda_n \ne \lambda_m$ whenever $n \ne m$. For the equation \eqref{eq:DEQ} to have a system of polynomial
solutions, it is necessary that $\a_2$ is a polynomial of degree at most $2$ and
$\a_1$ is a polynomial of degree at most 1. For $ n=0,1,\ldots$, the Hermite polynomials $H_n$ satisfy the equation
\begin{equation} \label{HermiteD}
  D_\CH y : = y''- 2 x y' = - 2 n y.
\end{equation}
For $\a > -1$ and $n =0,1,\ldots$, the Laguerre polynomials $L_n^{(\a)}$ satisfy the equation
\begin{equation}\label{LaguerreD}
      D_\CL^\a y: = x y'' + (\a+1 -x) y'  =- n  y.
\end{equation}
For $\a, \beta > -1$ and $n =0,1,\ldots$, the Jacobi polynomials $P_n^{(\a,\beta)}$ satisfy the equation
\begin{equation}\label{JacobiD}
   D_\CJ^{\a,\beta} y:= (1-x^2) y'' - (\a - \beta+ (\a+\beta+2)x ) y' =-  n (n+\a+\beta+1)y.
\end{equation}
The characterization of Bochner \cite{Bochner} shows that these are essentially the only systems
of orthogonal polynomials with respect to a positive definite inner product
$\la f, g \ra_{d\mu} = \int_\RR f(x) g(x) d\mu$ that satisfy \eqref{eq:DEQ}.
For the quasi-definite linear functional, the condition for the Laguerre family can be relaxed to
$- \a \in \RR \setminus  \NN$, and the condition for the Jacobi family can be relaxed to
$- \a, - \beta, -(\a+\beta+1) \in \RR \setminus  \NN$, and, up to a complex change of variables, the only other
system of orthogonal polynomials that satisfies \eqref{eq:DEQ} is the system of Bessel polynomials.

The equation \eqref{eq:DEQ} have solutions that are Sobolev orthogonal polynomials, since classical
orthogonal polynomials are Sobolev orthogonal polynomials themselves. One may ask the question if
there are other system of Sobolev orthogonal polynomials that are solutions of \eqref{eq:DEQ}. In
\cite{KL98}, the system of polynomials that are solutions of admissible \eqref{eq:DEQ} and are orthogonal
with respect to the bilinear form
$$
      \la f, g\ra = \int_{\RR} f(x) g(x) d\mu_0(x) + \int_{\RR} f'(x) g'(x) d\mu_1(x)
$$
is characterized when $d\mu_0$ and $d\mu_1$ are real, possibly signed, Borel measures on the real line.
The classical orthogonal polynomials and those obtained above through limiting process are the only systems
of polynomials in the positive definite case. In the Laguerre case, the limit is
$$
  \lim_{\a \to -1}  \frac{1}{\Gamma(\a+1)} \int_0^\infty f(x) x^\a e^{-x} dx = f(0),
$$
which holds if $f$ is bounded on $(0, \infty)$ and continuous at $x=0$, and it follows that the system of
polynomials $P_0(x)=1$ and $P_n(x): = L_n^{(-1)}(x)$ for $n \ge 1$ is orthogonal with respect to
the Sobolev inner product
$$
  \la f, g\ra_{\CL^{-1}} = \lambda_0 f(0) g(0) + \sum_{k=1}^m \lambda_k b^{k-1} \int_0^\infty f^{(k)}(x) g^{(k)}(x) x^{k-1} e^{-x} dx,
$$
where $\lambda_0, \lambda_1 > 0$ and $\lambda_k \ge 0$ for $2 \le k \le m$, and these polynomials are solutions of the differential
equation $D_{\CL}^{-1} y = - n y$. In the Jacobi case, the limits are
$$
  \lim_{\a \to -1}  (\a +1) \int_{-1}^1 f(x) (1-x)^\a  dx = f(1), \quad
   \lim_{\beta \to -1}  (\beta +1) \int_{-1}^1 f(x) (1+x)^\beta  dx = f(-1),
$$
which hold if $f$ is bounded on $(-1, 1)$ and continuous at $x=1$ or $x =-1$, respectively, and it follows that,
for $\beta > -1$, the system of polynomials $P_0(x) = 1$ and $P_n(x): = P_n^{(-1,\beta)}(x)$ for $n \ge 1$ is orthogonal
with respect to the Sobolev inner product
\begin{align*}
  \la f, g\ra_{\CJ^{-1,\beta}} = &  \lambda_0 f(1)g(1) +  \sum_{k=1}^m \lambda_k b^{(k-1,\beta)} \int_{-1}^1
    f^{(k)}(x) g^{(k)}(x) (1-x)^{k-1}(1+x)^{\beta+k} dx,
\end{align*}
where $\lambda_0, \lambda_1 > 0$ and $\lambda_k \ge 0$ for $2 \le k \le m$, and these polynomials are solutions of the differential
equation $D_{\CJ}^{-1,\beta} y = - n (n+\beta)y$. A similar result holds for $\a > -1$ and $\beta = -1$. The characterization
in \cite{KL98} was carried out for the quasi-definite case, which shows that, up to a complex linear change of variables,
the only systems are those listed above but with parameters being real numbers, except negative integers, and Bessel polynomials.

Essentially, solutions of the second order differential equation \eqref{eq:DEQ} do not lead to new families of Sobolev
orthogonal polynomials; see \cite{ARPP, KL98} for further results.

In \cite{Koek93}, R.Koekoek initiated a search for differential operators whose solutions are the Sobolev-Laguerre
polynomials with respect to the inner product
$$
   \la f, g \ra = \frac{1}{\Gamma(\a+1)} \int_0^\infty f(x)g(x) x^a e^{-x} dx + M f(0)g(0) + N f'(0)g'(0),
$$
where $\a > -1$, $M\ge 0$ and $N \ge 0$. These Sobolev orthogonal polynomials were found in \cite{KM93}
and they satisfy differential equations of the form
\begin{align*}
 M \sum_{i=0}^\infty a_i(x) y^{(i)}(x) +  N \sum_{i=0}^\infty b_i(x) y^{(i)}(x) & + MN  \sum_{i=0}^\infty c_i(x) y^{(i)}(x)\\
    &+ x y''(x) + (\a+1-x)y'(x) = - n y,
\end{align*}
where $a_i(x)$, $b_i(x)$ and $c_i(x)$ are polynomials, independent of $n$, of degree at most $i$. In the case
of $M=0$ and $N > 0$, the differential equation is of order $2 \a + 8$ if $\a$ is a nonnegative integer and of
infinite order otherwise. Furthermore, if $M > 0$ and $N > 0$, then the differential equation is of order $4 \a + 10$
if $\a$ is a nonnegative integer and of infinite order otherwise. For $\a = 0,1,2$, this was established in \cite{Koek93}
and the general case was proved in \cite{KKB98}.

This result has been extended more recently to the Sobolev orthogonal polynomials with respect to the bilinear form
\begin{align*}
  \la f, g\ra & = \int_0^\infty f(x) g(x) x^{\a - m}e^{-x}dx \\
       & +  \left(f(0), f'(0), \ldots, f^{(m-1)}(0)\right)M \left(g(0), g'(0), \ldots, g^{(m-1)}(0)\right)^T,
\end{align*}
where $m =2, 3, \ldots,$ and $M$ is an $m\times m$ matrix. For $m =2$, the differential operator was found
in \cite{Iliev} and for $\a \in \NN$ and $\a \ge m$, the differential operator was constructed in \cite{DI}.

\section{Zeros}
\setcounter{equation}{0}

It is well known that the zeros of orthogonal polynomials with respect to a probability measure supported on the real line
are real, simple, interlace and are located in the interior of the convex hull of the support of the measure. In the Sobolev setting,
some of the above properties are lost. We will specify the behavior of zeros for two types of Sobolev orthogonal polynomials.

\subsection{Sobolev inner products of the second type}
Here we consider the Sobolev inner product defined in \eqref{Sobolevtypeform}, which we state again as
\begin{align} \label{Sobolevtypeform2}
\la p, q \ra = & \int_I p(x)q (x) d\mu_{0}  \\
    & + \left(p(c), p'(c), \ldots, p^{(m)}(c)\right) A \left(q(c), q'(c), \ldots, q^{(m)}(c)\right)^T, \notag
\end{align}
where $I$ is an interval in $\RR$, and we denote by $S_n$ a Sobolev orthogonal polynomial of degree $n$ with
respect to this inner product.

We start with zeros of Sobolev orthogonal polynomials for the inner product
$$
\langle f, g\rangle _{S}= \int_{0}^{a} f(x) g(x) d\mu_{0}(x) + M_{1} f^{(r)}(0) g^{(r)}(0) + M_{2} f^{(s)}(0) g^{(s)} (0), \quad r< s,
$$
studied in \cite{deBr93}, where it was shown that the polynomial $S_{n}$ of degree $n$ has exactly $n-2$ real
and simple zeros in $(0,a)$, and if $s=r+1$, then $S_n$ has a pair of conjugate complex zeros $\alpha \pm i \beta$
with $\alpha <0$ and $\beta \neq 0$, whereas if $s>r+1$, then there is a real number $G$ such that
if $\min \{M_{1}, M_{2}\} > G$, then $S_n$ has either a pair of conjugate complex zeros as above or two real zeros
in $(-\infty, 0)$.

These results illustrate that the behavior of the zeros of Sobolev orthogonal polynomials for the inner
product \eqref{Sobolevtypeform2} is sensitive to small changes in the parameters of the inner product.

Next we describe the results for the simplest case of $m=1$ in \eqref{Sobolevtypeform2}. If
$A = \mathrm{diag} \{A_0,A_1\}$ and $m=1$, the polynomial $S_n$ is quasi-orthogonal of order $2$
with respect to the measure $(x-c)^2 d \mu_0$; that is, $ \int_I S_n(x) p(x) (x-c)^2 d\mu_0 = 0$ for all polynomials
$p$ of degree at most $n- 3$. The zeros of $S_n$ were analyzed in \cite{AlMarRezRon}, where it was proved that
the zeros of $S_{n}$ are real, simple, and at least $n-1$ of them belong to the interval $I$ of the support of
$d\mu_0$, when $c$ is an end point of the interval $I$. More precisely, let the zeros be denoted
by $z_{k,n}$, $k=1,2,\ldots, n$, and they are arranged in increasing order with $z_{n,n}$ being the largest one.
If $c = \sup I$ and neither $S_{n}(c) >0$ nor $S_{n}' (c)>0$, then the largest zero satisfies
$c \leq z_{n,n} < c + \frac{ c- z_{1,n}}{n-1}$ and $|z_{n,n} - c| < |z_{n-1,n} -c|$. Moreover, if $A_{0}\neq0$, then
$z_{n,n}- c < \frac{1}{2} (A_{1}/A_{0})^{1/2}$. Making a change of variable $x \mapsto -x$ in the inner product,
the same analysis applies to the case of $c = \inf I$.

The case when $c$ is outside of the convex hull of the support of $d \mu_{0}$ was analyzed in \cite{MarPePi}
under the assumption that $A_0 =0$. In this case, $S_{n}$ again has real, simple zeros and at least $n-1$ of them
belong to $I$. If $I$ is a bounded interval, then for $n$ large enough and $c > \sup I$, $S_{n}$ has a zero at the
right hand side of $c$. This zero converges to $c$ when $n$ tends to infinity, that is, $c$ becomes an attractor of
this exceptional zero. Furthermore, the zeros of $P_{n-1}(\cdot;d\mu_0)$ interlace with the zeros of $S_{n}$. Finally,
it was proved that every zero of $S_{n}$ is an increasing and bounded function of $N$.

When the matrix $A$ in \eqref{Sobolevtypeform2} is a non-diagonal matrix, the case $m=1$ was studied in
\cite{AlMarRez95}. Assuming that $A$ is positive definite, it was shown that if $I$ is a bounded interval
and $c >\sup I$, then there exists a positive integer $n_{0}$ such that, for $n\geq n_{0}$, the zeros of $S_{n}$
are real, simple, all except the largest one are located in $(\inf I, c)$, whereas the largest zero is greater than $c$.
Furthermore, if the non-diagonal entry $A_{0,1}$ is positive, then more can be said on the distance from the largest
zero to $c$,
\begin{enumerate}
\item  If $A_{0,0}A_{1,1} \geq 2 A_{0,1}^{2}$, then $z_{n,n} -c < \frac{A_{1,1}}{ 2 (A_{0,0}A_{1,1} - A_{0,1}^{2})^{1/2}}.$
\item  If  $A_{0,1}^{2}\geq A_{0,0}A_{1,1} \geq 2 A_{0,1}^{2},$ then $z_{n,n} -c < A_{0,1}/ A_{0,0}.$
\end{enumerate}

If $m>1$ and $A$ is a diagonal matrix, the zeros of $S_n$ were studied in \cite{ALR96}. In this case,
$S_{n}$ is quasi-orthogonal of order $m+1$ with respect to the measure $(x-c)^{m+1} d\mu_0$ as
in \eqref{SobolevtypeOPcon}. If $c$ does not belong
to the interior of the interval $I$ that supports $d\mu_0$, $S_{n}$ has at least $n-m-1$ zeros with odd multiplicity in the
interior of  $I$, whenever $n\geq m+1$. More generally, let $B(\mu_0)$ denote the support of the measure $d \mu_0$
and let $\Delta$ denote the convex hull of $B(\mu_0)$. If $\bar{n}$ denote the number of terms in the discrete part of \eqref{Sobolevtypeform2} whose order of derivative is less than $n$, then $S_{n}$ has at least $n-\bar{n}$ changes of
sign in the interior of $\Delta$. Furthermore, if $C_\a$ denote the open connected components of
$\mathrm{Int} \Delta \backslash B(\mu_0)$, then, for $n> m+1$, the number of zeros of the polynomial $S_{n}$
located in each component $C_\a$ is less than or equal to either $m+1$ or $m+2$, when $m$ is even or odd,
respectively.  Further refined results on the locations of zeros in $C_\a$ can be found in \cite{ALR96}.

\subsection{Sobolev inner products of the first type}
Zeros of Sobolev orthogonal polynomials with respect to the inner product
\begin{equation} \label{eq:ipd-sec9}
  \la f,g \ra_S = \int_{\RR} f(t)g(t)d \mu_0 +  \int_{\RR} f'(t)g'(t)d \mu_1
\end{equation}
were studied in the very beginning of Sobolev orthogonal polynomials.
In \cite{Alt}, Althammer proved that the zeros of $S_{n}$ are real, simple and are located in the interval $[-1,1]$
when $d\mu_{0}=dx$ and $d\mu_{1}= \lambda dx$, both supported on the interval $[-1,1]$.
Similar results were established by Brenner \cite{Br72} for $d\mu_{0}=e^{-x}dx$ and $d\mu_{1}=\lambda e^{-x}dx$, both
supposed on $[0,\infty)$. As already mentioned in Section 2, Althammer gave an example that shows the zeros of a
Sobolev orthogonal polynomials of degree $2$ can have a zero outside the support of the measures. An interesting
example in this regard appears in \cite{Mei94}, for which $d \mu_0 = dx$ on $[-1,3]$, $d\mu_1(x) = \lambda $ on
$[-1,1]$ and $d\mu_1(x)=1$ on $[1,3]$, where $\lambda > 0$. For $\lambda$ sufficiently large, it was shown that $S_{2n}$ has
exactly two real 
zeros, one in $(1,3)$ and one in $(-3,-1)$, and $S_{2n+1}$ has exactly one real zero, located in $(1,3)$. The example
was generalized to the situation that $d\mu_0$ has at least one point of increasing in $(1,a]$,  where $a >1$, 
 $d\mu_1 =
\lambda d \mu_0$ on $[-1,1]$ and $d\mu_1 =d\mu_0$ on $[1,a]$, with $\mu_2'(x) \ge c > 0$ on $(-1,1)$.

Zeros of Sobolev orthogonal polynomials are fairly well understood when $(d\mu_0, d\mu_1)$ is a coherent pair.
In this paragraph, we follow the notation in Subsection 5.1. For the inner product $\la \cdot,\cdot\ra_\lambda$ in
\eqref{eq:cohernt-ipd}, we denote the Sobolev orthogonal polynomials by $S_n(\cdot;\lambda)$ and denote the ordinary orthogonal
polynomial with respect to $d\mu_i$ by $P_n(\cdot; d\mu_i)$. Let $I_{0}$ denote the support of the measure $d\mu_{0}$.
It was proved in \cite{MeideBr02} that $S_{n}(\cdot;\lambda)$ has $n$ real, simple zeros and at most one of them is
outside $I_{0}.$ Furthermore, if $(z^{\lambda}_{k,n})_{k=1}^{n}$ and $(x_{k,n} (d\mu_{j}))_{k=1}^{n}$, denote the zeros of
$S_{n}(\cdot;\lambda)$ and $P_{n} (\cdot; d\mu_{j})$, respectively, then
\begin{enumerate}
 \item $z^{\lambda}_{1,n} < x_{1,n} (d\mu_{0}) < z^{\lambda}_{2,n} < \cdot \cdot < z^{\lambda}_{n,n} < x_{n,n} (d\mu_{0}), $
 \item  $z^{\lambda}_{1,n} < x_{1,n-1}( d\mu_{j}) < z^{\lambda}_{2,n} < \cdot \cdot <x_{n-1,n-1}( d\mu_{j}) < z^{\lambda}_{n,n},  \quad j=0,1.$
\end{enumerate}
Interlacing property between the zeros of two consecutive Sobolev orthogonal polynomials, that is,
$$
z^{\lambda}_{1,n} < z^{\lambda}_{1,n-1} < z^{\lambda}_{2,n} < \cdot \cdot <z^{\lambda}_{n-1,n-1} < z^{\lambda}_{n,n},
$$
holds when $d\mu_0$ and $d\mu_1$ are both Laguerre weight or both Jacobi weights.

If $\{d\mu_0, d\mu_1\}$ is a generalized coherent pair (see Section 5), however, few results on zeros of Sobolev
orthogonal polynomials are known. The case of (0,2) coherent pairs was analyzed in \cite{BrMe95}, where it
was shown that if $\lambda$ is large enough, then the zeros of $S_{n}(\cdot; \lambda)$ are real, simple and
interlace with the zeros of $P_{n} (\cdot; d\mu_{j}), j=0,1$, when $I_{0}= I_{1}$; furthermore, at most two of the zeros
are outside $I_{0}$. In \cite{Mor07} an illustrative example of $(2,0)$ coherent pair associated with Freud weights
was studied, in which the zeros of  $S_{n}(\cdot,\lambda)$ are shown to be real, simple and enjoy a separation
property.

For the inner product \eqref{eq:ipd-sec9}, potential theoretic method was used in \cite{GK} to study asymptotic
distribution of zeros and critical points Sobolev orthogonal polynomials. Under the assumption that
the support $I_j$ of the measure $d\mu_j$ is compact and regular for the Dirichlet problem in
$\overline{\CC} \setminus I_j$ for $j = 0, 1$ and further regularity assumption, it was proved that the
critical points, or zeros of the derivative of $S_n$, have a canonical asymptotic distribution supported on the real
line; that is, their counting measure  converges weakly to the equilibrium measure of the union
$\Sigma= I_{0} \cup I_{1}$. For the zeros of $S_n$, it was shown that the weak* limit of a subsequence of
their counting measure is supported on a subset of $\Sigma \cup \bar{V}$, where $V= \cup_{r>0} V_{r}$ and
$V_{r}$ denotes the union of those components of $\{z\in \CC: g_{\bar{\CC}\backslash \Sigma}(z; \infty) < r \}$
having empty intersection with $I_{0}$, in which $g_A(z;\infty)$ denotes the Green function of the set $A$ with
pole at infinity.  Further result on the weak* limit was given in the notion of balayage of a measure on
the compact set $K : =  \partial V \cup (\Sigma \backslash V)$. If $K = \Sigma$, then the weak* limit of the counting
measure of the zeros is the equilibrium measure on $\Sigma$.

For Sobolev inner products \eqref{SobolevtypeOPm}, asymptotic distributions of zeros of Sobolev orthogonal
polynomials and their derivatives were studied in  \cite{DOP09} in the setting that
$d\mu_k (x) = x^{\gamma} e ^{-\varphi_k(x)}dx$, where $\varphi_k$ are functions having polynomial growth at infinity
or iterative exponential functions, on $[0,\infty)$, which includes the Freud weights as special cases.

In another direction concerning zeros of Sobolev orthogonal products with respect to the inner product
\eqref{SobolevtypeOPm2}, global distribution of the zeros is related to the boundedness of the multiplication
operator $f \mapsto x f$ in the norms associated with the inner product in \cite{CD03}. More precisely, if this operator is
bounded, that is, there exists a positive real number $M$ such that $\langle xf, xf \rangle_ {S}\leq
M \langle f, f \rangle_ {S},$ then the set of zeros of the Sobolev orthogonal polynomials is contained in the ball
centered at the origin with radius $M$ and the vector of measures $(d\mu_{0}, \ldots, d\mu_{m})$ has compact
support. The paper also contains other qualitative results on the zeros. Notice, however, that this approach
does not give analytic properties, such as real, simple, interlacing,  of the zeros.

\section{Asymptotics}
\setcounter{equation}{0}

For ordinary orthogonal polynomials, three different types of asymptotics are considered: strong asymptotics,
outer ratio asymptotics and $n$-th root asymptotics. All three have been considered in the Sobolev setting
and we summarize most relevant results in this section.

\subsection{Sobolev inner products of the second type}
The first work on asymptotics for Sobolev orthogonal polynomials was carried out in \cite{MA93} for the
inner product $\langle f, g\rangle _{S}= \int_{-1}^{1} f(x) g(x) d\mu (x) + M_{1} f'(c) g' (c)$, where $c\in \RR$,
$M_{1}>0$ and the measure $d \mu_{0}$ belongs to the Nevai class $M(0,1)$. Using the outer ratio asymptotics
for the ordinary orthogonal polynomials $P_n(\cdot;d\mu_0)$ and the connection formula between
$P_{n}(\cdot; d\mu_{0})$ and the Sobolev orthogonal polynomials $S_n$, as in \eqref{SobolevtypeOPcon}, it was
shown that, if
$c\in \RR \setminus \mathrm{supp} \mu_{0},$ then
$$
\lim_{n \to \infty}\frac {S_{n}(z)}{P_{n}(z,d\mu_{0})}= \frac {(\Phi(z) -\Phi(c))^{2}}{ 2 \Phi(z) (z-c)}, \qquad
  \Phi(z):=  z + \sqrt{ z^{2} -1},
$$
locally uniformly outside the support of the measure, where $\sqrt{ z^{2} -1}>0 $ when $z>1$. Furthermore,
if the measure belongs to the Szeg\H{o} class, then the outer strong asymptotics for  $S_{n}$ can be
deduced in a straightforward way. If $c\in \mathrm{supp} \mu_{0}$, then $\lim\limits_{n\to \infty}\frac {S_{n}(z)}
{P_{n}(z; d\mu_{0})}= 1$ outside the support of the measure. 
By comparing the corresponding polynomials, such results can be deduced in a similar
manner if a mass point is added to the measure.

The first extension of the above results was carried out in \cite{AlMarRez95} for the Sobolev inner product
\eqref{Sobolevtypeform2} with a $2\times2$ matrix $A$. Under the same conditions on the measure, it was
proved that
$$
\lim_{n\rightarrow\infty}\frac {S_{n}(z)}{P_{n}(z, d\mu_{0})}= \left(\frac {(\Phi(z) -\Phi(c))^{2}}{ 2 \Phi(z) (z-c)}\right)^r,
\qquad r: = \mathrm{rank A},
$$
locally uniformly outside the support of the measure. The second extension appears in \cite{LMA95} for the
inner product
$$
\la f, g\ra =  \int  f(x) g(x) d\mu_{0} (x) + \sum_ {j=1}^{N} \sum_ {k=0}^{N_{j}} f^{(k)}(c_{j}) L_{j,k} (g;c_{j}),
$$
where $d\mu_{0}\in M(0,1)$, $\{c_{k}\}_{k=1} ^{N}$ are real numbers located outside the support of the measure,
and $L_{j,k} (g;c_{j})$ is the evaluation at $c_{j}$ of the ordinary differential operator $L_{j,k}$ acting on $g$ such
that $L_{j,N_{j}}$ is not identically zero for $j=1, 2, \cdot\cdot\cdot, N$. Assuming that the inner product is
quasi-definite so that a sequence of orthogonal polynomials exist, then on  every compact subset in
$\CC \setminus \mathrm{supp} d\mu_0$,
$$
\lim_{n\to \infty} \frac{S_{n} ^{(\nu)}(z)}{ P_{n} ^{(\nu)}(z, d\mu_{0})}
   = \prod_{j=1}^{m} \left( \frac {(\Phi(z) -\Phi(c))^{2}}{ 2\Phi(z) (z-c)} \right)^{I_{j}},
$$
where $I_{j}$ is the dimension of the square matrix obtained from the matrix of the coefficients of $L_{j,N_{j}}$
after deleting all zero rows and columns. The key idea for the proof is to reduce the Sobolev orthogonality in
this setting to ordinary quasi-orthogonality, so that the polynomial $S_{n}(x)$ can be expressed as a short
linear combination of polynomials $P_{m}(x; d\mu_{0})$ and then consider the behavior of the coefficients
in the linear combination.

If both the measure $d\mu_0$ and its support $\Delta$ are regular, then techniques from potential theory were
used in \cite{LagPi} to derive the $n$-th root asymptotics of the Sobolev orthogonal polynomials,
$$
\limsup_{n\to \infty} \|S_{n}^{(j)}\|_ {\Delta}^{1/n}= C(\Delta), \qquad j\geq 0,
$$
where $\|\cdot \|_ {\Delta}$ denotes the uniform norm in the support of the measure and $C(\Delta)$ is its
logarithmic capacity.

When the support of the measure in the inner product \eqref{Sobolevtypeform2} is unbounded, the analysis
has been focused on the case of the Laguerre weight function. The first study in \cite{AlvMor} considered
the case that $c=0$ and $A$ is a $2\times 2$ diagonal matrix (see also \cite{MarMor} for a survey on the
unbounded case).  Assuming that the leading coefficient of $S_n$ is normalized to be $\frac{(-1)^{n}}{n!}$,
the following results on the asymptotic behavior of $S_n$ were established:
\begin{enumerate}[\quad (1)]
\item (Outer relative asymptotics)
 $\lim\limits_{n\to \infty} \frac{S_{n} (z)}{ L_{n}^{(\alpha)}(z)}=1 $ uniformly on compact subsets of the
exterior of the positive real semiaxis.
\item (Outer relative asymptotics for scaled polynomials)
  $\lim\limits_{n\to \infty} \frac{S_{n} (nz)}{ L_{n}^{(\alpha)}(nz)}=1$ uniformly on compact subsets of the exterior of $[0,4]$.
\item (Mehler-Heine formula)  $\lim\limits_{n\to \infty} \frac{S_{n} (z/n)}{n^{\alpha}}= z^{-\alpha/2} J_{\alpha+ 4} (2 \sqrt{z})$ uniformly on compact subsets of the complex plane, assuming that $\mathrm{rank} A=2$.
\item (Inner strong asymptotics)
$$
  \frac{S_{n} (x)}{n^{\alpha/2}}= c_{3}(n) e^{x/2} x^{-\alpha/2}J _{\alpha+4} (2 \sqrt{(n-2)x}) + O\left( n^{- \min\{\alpha + 5/4, 3/4\}}\right)
$$
on compact subsets of the positive real semiaxis, where $\lim\limits_{n\to \infty}c_{3}(n)=1.$
\end{enumerate}
The case that $\mathrm{rank}\, A=1$ is also studied, we only state the results when $A$ has full rank for sake of simplicity.
Finally, if the point $c$ is a negative real number, then the following outer relative asymptotics was established in
\cite{FejHueMar},
$$
  \lim_{n\to \infty} \frac{S_{n} (z)}{ L_{n}^{(\alpha)}(z)} = \left(\frac{\sqrt{-z} - \sqrt{-c}}{\sqrt{-z} + \sqrt{-c}}\right)^r,
  \qquad r = \mathrm{rank} A,
$$
uniformly on compact subsets of the exterior of the real positive semiaxis. 

When $c =0$ and $A$ is a non-singular diagonal matrix of size $m+1$, the following asymptotic properties of the
Sobolev orthogonal polynomials with respect to the inner product \eqref{Sobolevtypeform2} were obtained in
\cite{AlMoPeRe},
\begin{enumerate}[\quad (1)]
\item (Outer relative asymptotics)
For every $\nu \in \NN$, $\lim\limits_{n\to \infty} \frac{S_{n}^{(\nu)} (z)}{ (L_{n}^{(\alpha)})^{(\nu)}(z)}=1$ uniformly on compact subsets of the exterior of the positive real semiaxis.

\item (Mehler-Heine formula) $\lim\limits_{n\to \infty} \frac{(-1)^{n}}{n!}\frac{S_{n} (z/n)}{n^{\alpha}}= (-1) ^{m+1} z^{-\alpha/2} J_{\alpha+ 2m +2} (2 \sqrt{z})$ uniformly on compact subsets of the complex plane.
\end{enumerate}
It should be pointed out that the above Mehler-Heine formula cannot be directly deduced from
the connection formula \eqref{SobolevtypeOPcon}. 
As an interesting consequence of the Mehler-Heine formula and the Hurwitz Theorem, the local behavior of zeros of
these Sobolev orthogonal polynomials can be deduced. 

\subsection{Sobolev inner products of the first type}
We first consider the case of coherent pair. Let $\{\mu_{0}, \mu_{1}\}$ be a coherent pair of measures and
$\mathrm{supp} \mu_{0}= [-1,1]$. Then the outer relative asymptotic relation for the Sobolev orthogonal polynomials
with respect to \eqref{eq:ipd-sec9} in terms of orthogonal polynomials $P_n(\cdot; d\mu_1)$ is (\cite{MMPP})
$$
\lim\limits_{n\to \infty} \frac{S_{n} (z)}{P_{n}(z;d\mu_1)} = \frac{2}{\Phi'(z)}, \qquad
\Phi(z):= z+\sqrt{z^2-1},
$$
where $\sqrt{z^2-1} > 0$ when $z > 1$, uniformly on compact subsets of the exterior of the interval $[-1,1]$.
The key idea of the proof is to show that the coefficient $b_n(\lambda)$ in \eqref{eq:SOP-OP-A1} satisfies
$\lim\limits_{n\to\infty} b_{n}(\lambda)=0$, which is a consequence of the inequality
\begin{equation*}
\pi_{n} + \lambda n^{2}\tau_{n-1} \leq \|S_{n}(\cdot, \lambda)\|_{\lambda}^2 \leq  \pi_{n} + \hat{a}_{n-1}^{2} \pi_{n-1} + \lambda n^{2} \tau_{n-1},
\end{equation*}
where $\pi_{n} := \| P_{n}(\cdot, d\mu_{0})\|_{d\mu_{0}}^{2}$ and $\tau_{n}:= \| P_{n}(\cdot, d\mu_{1})\|_{d\mu_{1}}^{2}$,
on the weighted $L^{2}$ norms.

When the measures $\mu_{0}$ and $\mu_{1}$ are absolutely continuous and belong to the Szeg\H{o} class, the
Bernstein-Szeg\H{o} theory is applied in \cite{Mart} to derive the outer relative asymptotics
$\lim\limits_{n\to\infty} \frac{S_{n} (z)}{ P_{n}(z,d\mu_{1})}= \frac{2}{\Phi'(z)}$ locally uniformly outside $[-1,1]$. As a straightforward consequence, the strong outer asymptotics follows. If $\mu_{0}$ is any finite Borel measure supported in $[-1, 1]$ in \cite{Mart}, then the outer relative asymptotics $\lim\limits_{n\to\infty}
 \frac{S'_{n} (z)}{ n P_{n-1}(z,d\mu_{1})}= 1$ holds, which illustrates the role of the measure involved in the
nonstandard part of the Sobolev inner product.

An extension of the above results for the inner product \eqref{SobolevtypeOPm2} was carried out in \cite{MartPij},
where $d\mu_{m}$ is assumed to be absolutely continuous and belongs to the Szeg\H{o} class, whereas
other measures are assumed to be positive and supported in $[-1,1]$. When the supports of the measures involved
in the Sobolev inner product satisfy certain nesting property and the measure $d\mu_{0}$ is regular, the $n$-root
asymptotic behavior is established. More specifically, for every nonnegative integer number $j$,
\begin{enumerate}[\quad (1)]
\item $\limsup\limits_{n\to\infty} | S^{(j)}_{n}(z)|^{1/n} = C(\Delta) e^{g_{\Delta}(z, \infty)}$
for every $z\in \CC$ up to a set of capacity zero;

\item $\lim\limits_{n\to \infty} |S^{(j)}_{n}(z)|^{1/n} = C(\Delta) e^{g_{\Delta}(z, \infty)}$ uniformly on each compact subset
of the complex plane outside the disk centered at the origin with radius defined in terms of the Sobolev norm of the
multiplication operator,
\end{enumerate}
where $g_{\Delta} (0, \infty)$ is the Green function of $\Delta$ with the infinite as a pole (\cite{LagPi}).

For measures of coherent pairs that have unbounded support, asymptotic properties of the corresponding
Sobolev orthogonal polynomials have been extensively studied in the literature (see \cite{MarMor} for an overview).
The outer relative asymptotics, the scaled outer asymptotics as well as the inner strong asymptotics of such polynomials
have been considered for all families of coherent pairs and symmetrically coherent pairs. It should be point out
that the outer relative asymptotics involves the parameter $\lambda$ in the inner product \eqref{eq:cohernt-ipd},
which does not appear in the case of bounded supports. Nevertheless, the technicalities in the unbounded case
are similar to those in the bounded case, the main difficulty is on estimates for the coefficients that appear
in the relations between Sobolev and ordinary orthogonal polynomials.

For instance, in the case of symmetrically coherent Hermite case, that is, when either $d\mu_{0}= e^{-x^{2}}$ and 
$d\mu_{1}= \frac{1}{x^{2}+ \xi^{2}} e^{-x^{2}}$ or  $d\mu_{0}= x^{2}+ \xi^{2} e^{-x^{2}}$ and $d\mu_{1}= e^{-x^{2}}$, 
it holds, respectively,
\begin{enumerate}[\quad (1)]
\item $\lim\limits_{n\to \infty} \frac{S_{n}(z,\lambda)}{H_{n}(z)}= \theta (\lambda)$ uniformly on compact subsets outside 
the real line;
\item $\lim\limits_{n\to \infty} \lfloor \f{n}{2} \rfloor^{1/2}\frac{S_{n}(z,\lambda)} {H_{n}(z)} = \theta (\lambda) (\mp z + |\xi|),$ if $z\in \CC_{\pm}$, uniformly on compact subsets of the half planes $\CC_{\pm}= \{ x \pm i y, y>0 \}$, respectively,
\end{enumerate}
where $\theta (\lambda):= \frac{\Phi(2\lambda +1)}{\Phi(2\lambda +1)-1}$, as shown in \cite{MarMor}.

The case $d\mu_{0}= d\mu_1= e^{-x^{4}}$ of the inner product \eqref{eq:cohernt-ipd}, or the (2,0) self-coherent pair with
Freud weight, was studied in \cite{CaMarMor}, where the connection between Sobolev and ordinary orthogonal
polynomials is given by $P_{n}(x;d\mu)= S_{n}(x;\lambda) + c_{n-2}(\lambda) S_{n-2}(x;\lambda)$. In order to deduce its
outer relative asymptotics, it is necessary to estimate the values of $c_{n}(\lambda)$ for $n$ sufficiently large. It turns
out that $\lim\limits_{n\to \infty} n^{-1/2} c_{n}(\lambda)= \frac{1} {6\sqrt{3}}$, which leads to, in a straightforward way,
$\lim\limits_{n\to \infty} \frac{S_{n}(z,\lambda)}{P_{n}(z, d\mu_{0})}= \frac{3}{2}$ uniformly on compact subsets
outside the real line. The outer relative asymptotics, with appropriate scaling, was also obtained in \cite{CaMarMor},
taking into account the Rakhmanov-Maskhar-Saff constant for these Freud polynomials.

For asymptotics, few examples of generalized coherent pairs of measures have been considered in the literature.
The Freud-Sobolev case considered above motivated further study of Sobolev inner products with exponential weights.
Let $W(x) = e^{-Q(x)}$, where $Q$ is a continuous even function in $\RR$ such that  $Q''$ is continuous in $(0,\infty)$
and $Q'>0$ in $(0,\infty)$, and $\alpha \leq \frac{x Q''(x)}{Q'(x)} \leq \beta$ for some $\beta > \alpha > 0$. The asymptotic
behavior of Sobolev orthogonal polynomials for the inner product \eqref{eq:ipd-sec9} with
$d\mu_{0}(x) = (\psi (x) W(x))^{2} dx$ and $d\mu_{1}(x)= \lambda W^{2} (x) dx$, where $\psi \in L_{\infty} (\RR)$,
was studied in \cite{GeLuMar}. The derivative of the orthonormal Sobolev polynomials $s_{n}$ behaves as $\lambda^{-1/2} p_{n-1}(.; W^2)$, the orthonormal polynomials with respect to the measure $W^{2} (x)dx$, in the sense of the $L_{2}$ asymptotics, that is, $\| s'_{n} (\cdot)- \lambda ^{-1/2} p_{n-1} (\cdot; W^{2}) \| _{W^{2} dx}= O (\frac{a_{n}}{n})$,
where $a_{n}$ is the Mhaskar-Rakhmanov-Saff number for $Q$. On the other hand, an uniform bound for the corresponding scaled polynomials  is also deduced with the help of a simple Nikolskii inequality.  As in the bounded case, 
the measure $d\mu_{1}$ plays a key role in the asymptotic behavior of the Sobolev orthonormal polynomials.

\section{Sobolev orthogonal polynomials of several variables}
\setcounter{equation}{0}

In contrast to one variable, Sobolev orthogonal polynomials of several variables are studied only recently.
In this section we report what has been done in this direction.

\subsection{Orthogonal polynomials of several variables}

For $x \in \RR^d$ and $\a \in \NN_0^d$, the (total) degree of the monomial $x^\a$ is, by definition,
$|\a|:= \a_1+\cdots+ \a_d$. Let $\Pi_n^d$ denote the space of polynomials of total degree $n$ in
$d$-variables. It is known that $\dim \Pi_n^d = \binom{n+d}{n}$. Let $\Pi^d$ denote the space of all polynomials
in $d$-variables. Let $\la\cdot, \cdot \ra$ be an inner product defined on $\Pi^d \times \Pi^d$. A polynomial
$P \in \Pi_n^d$ is orthogonal if $ \la P, q\ra = 0$ for all $q \in \Pi_{n-1}^d$. For $ n \in \NN_0^d$, let $\CV_n^d$
denote the space of polynomials of degree $n$. Then $\dim \CV_n^d = \binom{n+d-1}{n}$. In contrast to one-variable,
the space $\CV_n^d$ can have many different bases when $d \ge 2$. Moreover, the elements in $\CV_n^d$ may not
be orthogonal to each other.

For the structure and properties of orthogonal polynomials in several variables, we refer to \cite{DX}. In the following,
we describe briefly two families of orthogonal polynomials as examples.

\medskip\noindent
{\bf Example 1.} For $i=1,2$, let $w_i$ be a weight functions defined on an interval $[a_i,b_i]$ of the real line and let
$p_n(w_i; \cdot)$ be the orthogonal polynomial of degree $n$ with respect to $w_i$. With respect to the weight
function $W(x,y) := w(x) w(y)$ on $[a_1,b_1]\times [a_2,b_2]$, one family of mutually orthogonal basis of $\CV_n^2(W)$
is given by
\begin{equation} \label{eq:productOP}
  P_k^n(x,y) = p_k(w_1;x) p_{n-k}(w_2;y), \qquad 0 \le k \le n.
\end{equation}
Because $\Pi_n^d$ is defined in terms of total degree, which is different from the tensor product $\Pi_n \times \Pi_n$,
every element of $\CV_n^2$ is of total degree $n$.

\medskip

For our second example, we need the definition of spherical harmonics. Let $\CP_n^d$ denote the space
of homogeneous polynomials of degree $n$ in $d$-variables. A polynomial $Y$ is a spherical harmonic of
degree $n$ if $Y \in \CP_n^d \cap \ker \Delta$. Let $\CH_n^d$ denote the space of spherical harmonics of
degree $n$. It is known that
$$
  a_n^d: = \dim \CH_n^d = \binom{n+d-1}{n} - \binom{n+d-3}{n-2}.
$$
The elements of $\CH_n^d$ are orthogonal to lower degree polynomials with respect to the inner product
$$
   \la f,g \ra_{\sph}: =   \int_{\sph} f(\xi) g(\xi) d\s(\xi),
$$
where $d\s$ denotes the surface measure on $\sph$. An orthonormal basis can be constructured explicitly
in terms of the Geganbauer polynomials in spherical coordinates. For properties of harmonic polynomials and
their relations to orthogonal polynomials, we refer to \cite{DaiX, DX} and references therein.

\medskip\noindent
{\bf Example 2.}
For $\mu > -1$, let $\varpi_\mu(x) = (1-\|x\|^2)^{\mu-1/2}$ be the weight function defined on the unit ball
$\ball = \{x \in \RR^d: \|x\|\le 1\}$, where $\|x\|$ denotes the Euclidean norm of $x \in \RR^d$. Orthogonal
polynomials with respect to $\varpi_\mu$ can be given in several different formulations. We give one basis of
$\CV_n^d(\varpi_\mu)$ in terms of the Jacobi polynomials and spherical harmonics in the spherical coordinates
$x = r \xi$, where $0 < r\le 1$ and $\xi \in \sph = \{x: \|x\|=1\}$. For $0 \le j \le n/2$ and $1 \le j \le a_{n-2j}^d$,
define
\begin{equation} \label{eq:ball-basis}
  P_{j,\nu}^n(x) := P_{j}^{(\mu, n-2j + \frac{d-2}{2})}(2\,\|x\|^2 -1)\, Y_\nu^{n-2j}(x),
\end{equation}
where $\{Y_\nu^{n-2j}: 1 \le \nu \le a_{n-2j}^d\}$ is an orthonormal basis of $\CH_{n-2j}^d$. Then the
set $\{P_{j,\ell}^{\mu,n}(x): 0 \le j \le \tfrac{n}{2}, \,1 \le \ell \le a_{n-2j}^d \}$ is a mutually orthogonal basis
of $\CV_n^d(\varpi_\mu)$. The elements of $\CV_n^d$ are eigenfunctions of a second order
differential operator $\CD_\mu$. More precisely, we have
\begin{align} \label{eq:Bdiff}
    \CD_\mu P =  -(n+d) (n + 2 \mu)P, \qquad \forall P \in \CV_n^d(\varpi_{\mu}),
\end{align}
where
\begin{equation*}
  \CD_\mu := \Delta  - \sum_{j=1}^d \frac{\partial}{\partial x_j} x_j \left[
  2 \mu  + \sum_{i=1}^d x_i \frac{\partial  }
  {\partial x_i} \right].
\end{equation*}

\subsection{Sobolev orthogonal polynomials on the unit ball}
At the moment of this writing, we know more about Sobolev orthogonal polynomials on the unit ball than
on any other domain.

The first work in this direction deals with the inner product  \cite{X06}
$$
  \la f, g \ra_{\Delta} : = \int_{\ball} \Delta \left[(1-\|x\|^2)f(x) \right]\Delta \left[(1-\|x\|^2)g(x) \right] dx,
$$
which arises from numerical solution of the Poisson equation \cite{AH05}. The geometry of the ball and
\eqref{eq:ball-basis} suggests that one can look for a mutually orthogonal basis of the form
\begin{equation} \label{eq:ball-basis2}
   q_j(2 \|x\|^2 -1) Y_\nu^{n-2j}(x), \qquad Y_\nu^{n-2j} \in \CH_{n-2j}^d,
\end{equation}
where $q_j$ is a polynomial of degree $j$ in one variable. Such a bias was constructed in \cite{X06} for the
space $\CV_n^d(\Delta)$ with respect to $\la \cdot ,\cdot \ra_\Delta$. As a result, it was shown that
$$
   \CV_n^d(\Delta) = \CH_n^d \oplus (1-\|x\|^2) \CV_{n-2}(\varpi_2).
$$

The next inner product considered on the ball, which should be, in retrospect, the first one being considered,
is defined by
$$
  \la f,g\ra_{-1} : =   \lambda \int_{\ball} \nabla f(x) \cdot  \nabla g(x) dx + \int_{\sph} f(\xi)g(\xi) d\s(\xi),
$$
where $\lambda > 0$. An alternative is to replace the integral over $\sph$ by $f(0)g(0)$. A basis of the form
\eqref{eq:ball-basis2} was constructed explicitly in \cite{X08} for the space $\CV_n^d(\Delta)$ with
respect to $\la \cdot ,\cdot \ra_{-1}$, from which it follows that
\begin{equation}\label{eq:decom-1}
   \CV_n^d(\varpi_{-1}) = \CH_n^d \oplus (1-\|x\|^2) \CV_{n-2}(\varpi_1).
\end{equation}
The main part of the base, those in $(1-\|x\|^2) \CV_{n-2}(\varpi_1)$, can be given in terms of
the Jacobi polynomials $P_n^{(-1, b)}$ of negative index, which explains why we used the notation
$\varpi_{-1}$. Another interesting aspect of this case is that the polynomials in $\CV_n^d(\varpi_{-1})$
are eigenfunctions of the differential operator $\CD_{-1}$, the limiting case of \eqref{eq:Bdiff}.

For $k \in \NN$, the equation $\CD_{-k} Y = \lambda_n Y$ of \eqref{eq:Bdiff} was studied in \cite{PX09}, where
a complete system of polynomial solutions was determined explicitly. For $k \ge 2$, however, it is not known
if the solutions are Sobolev orthogonal polynomials. Closely related to the case of $k=2$ is the following
inner product
$$
   \la f, g \ra_{-2}: =  \lambda \int_{B^d} \Delta f(x) \Delta g(x) dx +
                        \int_{S^{d-1}} f(x) g(x) d\s, \quad \lambda > 0.
$$
An explicit basis for the space $\CV_n^d(\varpi_{-2})$ of the Sobolev orthogonal polynomials with respect to
$\la \cdot,\cdot \ra_{-2}$ was constructed in \cite{PX09}, from which it follows that
\begin{equation}\label{eq:decom-2}
    \CV_n^d(\varpi_{-2})  = \CH_n^d  \oplus  (1-\|x\|^2)\CH_{n-2}^d \oplus
       (1-\|x\|^2)^2 \CV_{n-4}^d(\varpi_2).
\end{equation}
The main part of the base, those in $(1-\|x\|^2)^2 \CV_{n-2}(\varpi_2)$, can be given in terms of
the Jacobi polynomials $P_n^{(-2, b)}$ of negative index

It turns out that the Sobolev orthogonal polynomials for the last two cases can be used in the study of the
spectral method for numerical solution of partial differential equations. This connection was established in
\cite{LX13}, where, for $s \in \NN$, the following inner product in the Sobolev space  $W_p^s(\BB^d)$ is
defined,
\begin{align} \label{eq:ipd-s}
 & \la f, g\ra_{-s}:= \la \nabla^{s}f, \nabla^{s} g \ra_{\ball}
  + \sum_{k=0}^{\rhow-1}\lambda_{k} \la  \Delta^{k}f, \Delta^{k}g\ra_{\sph},
\end{align}
where  $\lambda_{k}, \, k=0,1\dots,\rhow-1$, are positive constants, and
\begin{align*}
\nabla^{2m} := \Delta^m \quad\hbox{and}\quad  \nabla^{2m+1} := \nabla \Delta^{m},  \qquad m=1,2,\dots.
\end{align*}
For $s > 2$, the space $\CV_n^d(\varpi_{-s})$ associated with $\la \cdot, \cdot \ra_{-s}$ cannot be decomposed
as in \eqref{eq:decom-1} and \eqref{eq:decom-2}. Nevertheless, an explicit mutually orthogonal basis was
constructed in \cite{LX13}, which requires considerable effort, and the basis uses extension of the Jacobi
polynomials $P_n^{(\a,\beta)}$ for $\a, \beta \in \RR$ that avoids the degree reduction when $-\a-\beta-n \in  \{0,1,\ldots, n\}$.
The main result in \cite{LX13} establishes an estimate for the polynomial approximation in the Sobolev space
$\CW_p^s(\ball)$, the proof relies on the Fourier expansion in the Sobolev orthogonal polynomials associated
with \eqref{eq:ipd-s}.

Another Sobolev inner product considered on the unit ball is defined by
$$
 \la f ,g \ra =  \int_{\BB^d} \nabla f(x)\cdot \nabla g(x) W_\mu (x) dx + \lambda \int_{\BB^d} f(x) g(x) W_\mu(x) dx,
$$
which is an extension of the Sobolev inner product of the coherent pair in the case of the Gegenbauer weight
of one-variable. A mutually orthogonal basis was constructed in \cite{PPX13}, which has the form of
\eqref{eq:ball-basis2} but the corresponding $q_j$ is orthogonal with respect to a rather involved Sobolev
inner product of one variable.

 \def\bg{{\boldsymbol{\large {\g}}}}
 \def\one{{\mathbf 1}}
 \def\zero{{\mathbf 0}}
  \def\nb{{\mathbf n}}

\subsection{Sobolev orthogonal polynomials on the simplex}
Let $T^d$ be the simplex of $\RR^d$ defined by
$$
  T^d := \{x \in \RR^d: x_1 \ge 0, \ldots, x_d \ge 0, 1-|x| \ge 0 \},
$$
where $|x| : =x_1+\cdots + x_d$. The classical weight function on $T^d$ is defined by
\begin{equation}\label{Weight}
    \varpi_\bg (x): = x_1^{\g_1} \cdots x_d^{\g_d} (1- |x|)^{\g_{d+1}}, \quad x \in T^d,
\end{equation}
where $\g_i$ are real numbers, usually assumed to satisfy $\g_i > -1$ to ensure the integrability
of $\varpi_\bg$ on $T^d$. Let $c_\bg = 1 \Big / \int_{T^d} \varpi_\bg(x) dx$ denote the normalization
constant. We consider the inner product
$$
       \la f, g\ra_\bg : = c_\bg \int_{T^d} f(x) g(x) \varpi_\bg(x) dx.
$$
The space $\CV_n^d(\varpi_\bg)$ of orthogonal polynomials of degree $n$ for this inner product contains
several explicit bases, which have been studied extensively (cf. \cite{DX}). One particular basis is given
by the Rodrigues type formula
\begin{equation} \label{RodrigueT}
 P^{\bg}_\nb (x) := x^{-\g} (1-|x|)^{-\g_{d+1}} \frac{\partial^{|\nb|}} {\partial x^{\nb}}
     \left[x^{\g + \nb} (1-|x|)^{\g_{d+1} + |\nb|} \right],
\end{equation}
where $\frac{\partial^{|\nb|}} {\partial x^{\nb}}=\frac{\partial^{|\nb|}} {\partial x_1^{n_1} \cdots
\partial x_d^{n_d}}$ and $\nb \in \NN_0^d$. Furthermore, it is know that polynomials in
$\CV_n^d(W_\bg)$ are eigenfunctions of the differential operator
$$
 L_\bg P : =  \sum_{i=1}^d x_i(1 -x_i)   \frac {\partial^2 P} {\partial x_i^2}   -
 2 \sum_{1 \le i < j \le d} x_i x_j \frac {\partial^2 P}{\partial x_i
 \partial x_j}
  + \sum_{i=1}^d \left( \g_i +1-(|\bg|+d+1) x_i \right) \frac {\partial P}{\partial x_i},
$$
where $|\bg|:= \g_1 + \cdots + \g_{d+1}$; more precisely,
\begin{equation} \label{diff-eqn}
    L_\bg P  =  -n  \left(n+|\bg| +d \right)P, \qquad \forall P \in \CV_n^d(\varpi_\bg).
\end{equation}
If some or all $\g_i$ are equal to $-1$, the weight function becomes singular but the equation \eqref{diff-eqn}
still has a full set of polynomial solutions. In \cite{AX13}, these solutions were given explicitly and were shown
to be the Sobolev orthogonal polynomials with respect to explicitly given inner products. For $d=2$, the simplex
is the triangle $T^2$, writing the weight function as $\varpi_{\a,\beta,\g}(x) = x^\a y^\beta (1-x-y)^\g$, then the inner products
are of the form
\begin{align*}
 \la f, g \ra_{\a,\beta,-1}: = &\,  \int_{T^2} \left [x  \partial_x f(x,y) \partial_x g(x,y) + y \partial_y f(x,y) \partial_y g(x,y) \right]
               x^\a y^\beta  dxdy \\
        &\, + \lambda_1 \int_0^1 f(x,1-x)g(x,1-x) x^\a(1-x)^\beta dx \\
 \la f, g \ra_{\a,-1,-1}: = &\,   \int_{T^2} \partial_y f(x,y) \partial_y g(x,y)  x^\a dxdy, \\
        &\, + \lambda_1 \int_0^1 \partial_x f(x,0) \partial_x g(x,0) x^{\a+1} dx + \lambda_{1,0} f(1,0) g(1,0) \\
 \la f, g \ra_{-1,-1,-1}: = &\,   \int_{T^2} \partial_{xy} f(x,y) \partial_{xy} g(x,y) (1-x-y) dxdy, \\
        &\,  + \lambda_1 \int_0^1 \partial_x f(x,0) \partial_x g(x,0) dx + \lambda_2 \int_0^1 \partial_y f(0,y) \partial_y g(0,y) dy \\
        & \, + \lambda_{1,0} f(1,0) g(1,0) + \lambda_{0,1} f(0,1) g(0,1) + \lambda_{0,0} f(0,0) g(0,0).
\end{align*}
The first case can be deduced by taking limit $\g \to -1$ in the classical inner product with respect to $\varpi_{\a,\beta,\g}$,
as observed in \cite{BDFP}. Identifying the correct form of the inner product is a major step. For $d > 2$, one
needs to consider the lower dimensional faces of the simplex $T^d$. Fortunately, restrictions of the
polynomials in \eqref{RodrigueT} remain orthogonal polynomials on the faces.

Besides those considered in \cite{AX13}, no other family of Sobolev orthogonal polynomials has
been studied on the simplex.

\subsection{Sobolev orthogonal polynomials on product domain}

On the product domain $[a_1,b_1]\times [a_2,b_2]$ of $\RR^2$, define the product weight function
$$
      \varpi(x_1, x_2) = w_1(x_1) w_2(x_2),
$$
where $w_i$ is a weight function on $[a_i,b_i]$ for $i=1,2$. With respect to $\varpi$, we consider the
Sobolev inner product
$$
  \la f, g\ra_S: = \int_{[a,b]^2} \nabla f(x,y) \cdot \nabla g(x,y) \varpi(x,y) dx dy + \lambda f(c_1,c_2)g(c_1,c_2),
$$
where $\nabla f = (\partial_x f, \partial_y g)$, $\lambda > 0$ and $(c_1,c_2)$ is a fixed point in $\RR^2$.

A moment reflection shows that orthogonal polynomials with respect to this inner product are no longer products
as that of \eqref{eq:productOP} in general. Two cases are consider in \cite{FMPPX}. The first one is the product
Laguerre weight for which
$$
  \la f, g\ra_S: = \int_0^\infty \int_0^\infty \nabla f(x,y) \cdot \nabla g(x,y) w_\a(x) w_\beta(y) dx dy + \lambda_k f(0,0) g(0,0),
$$
where $w_\a(x) = x^\a e^{-x}$. The Sobolev orthogonal polynomials are related to the polynomials
$Q_{j,m}^{\a,\beta}$ defined by
$$
  Q_{j,m}^{\a,\beta}(x,y) := Q_{m-j}^\a (x) Q_j^\beta(y) \quad \hbox{with}\quad
    Q_n^\a(x) := \wh L_n^{(\a)}(x) + n \wh L_{n-1}^{(\a)}(x),
$$
where $\wh L_n^{(\a)}$ denotes the $n$-th monic Laguerre polynomial. 
The polynomial $Q_n^\a$ is monic and it satisfies $ \f{d}{dx} Q_n^\a(x) = n \wh L_{n-1}^{(\a)}(x)$.  For $0 \le k \le n$, let
$S_{n-k,k}^{\a,\beta}(x,y) = x^{n-k} y^k + \cdots$ be the monic Sobolev orthogonal polynomials of degree $n$. Define
the column vectors
$$
 \QQ_n^{\a,\beta} := (Q_{0,n}^{\a,\beta}, \ldots, Q_{n,n}^{\a,\beta})^T \qquad  \hbox{and}\qquad
   {\mathbb S}_n^{\a,\beta}:= (S_{0,n}^{\a,\beta}, \ldots, S_{n,n}^{\a,\beta})^T.
$$
It was shown in \cite{FMPPX} that there is a matrix $B_{n-1}$ such that
$$
    \QQ_n^{\a,\beta} = {\mathbb S}_n^{\a,\beta} + B_{n-1} {\mathbb S}_{n-1}^{\a,\beta}
$$
and the matrix $B_{n-1}$ and the norm  $\la {\mathbb S}_n^{\a,\beta}, {\mathbb S}_n^{\a,\beta}\ra_S$ can both be computed by one
recursive algorithm.

The above construction of orthogonal basis for the product domain works if $w_1$ and $w_2$ are self-coherent.
The case that both are the Gegenbauer weight functions was given as a second example in \cite{FMPPX}.

\subsection{Miscellaneous results}

Sobolev orthogonal polynomials in two variables that satisfy second order partial differential equations
are discussed In \cite{LL06}. The paper, however, contains few concrete examples.

In \cite{PI}, a large family of commutative algebras of partial differential operators invariant under rotations,
called Krall-Jacobi algebra, is constructed and analyzed. The study leads naturally to the Sobolev orthogonal
polynomials with respect to an inner product on the unit ball that involves spherical Laplacian.

For inner product that contains additional point evaluations of functions, as those discussed in Section 7, the
Krall type construction of orthogonal polynomials can be extended to several variables, as shown in \cite{DFPPX}.
The same holds true if the point evaluations involve derivatives, see \cite{DPP} for an example.

\section{Orthogonal expansions in Sobolev orthogonal polynomials}
\setcounter{equation}{0}

Let $\{p_n\}_{n \ge 0}$ be a system of orthogonal polynomials with respect to an inner product $\la \cdot,\cdot \ra$.
The Fourier orthogonal expansion of a function $f$ in $\{p_n\}_{n \ge 0}$ is defined by
$$
   f(x) \sim \sum_{n=0}^\infty \wh f_n p_n(x), \qquad  \wh f_n = \frac{1}{\sqrt{\la p_n,p_n\ra}} \la f, p_n\ra.
$$
The $n$-th partial sum $S_n f$ if defined by
$$
   S_n f(x):= \sum_{k=0}^n \wh f_k p_k(x) = \la f, K_n(x,\cdot)\ra,
$$
where $K_n(\cdot,\cdot)$ is the reproducing kernel of the space of polynomials of degree at most $n$, defined by
$$
   K_n(x,y) = \sum_{k=0}^n  \frac{1}{{\la p_k,p_k\ra}} p_k(x) p_k(y).
$$

The Fourier expansions in terms of classical orthogonal polynomials have been studied extensively in the literature.
One of the essential tools is the Christoffel-Darboux formula, which gives a closed formula for the kernel
$K_n(\cdot,\cdot)$. For Sobolev orthogonal polynomials, however, the Christoffel-Darboux formula no longer holds,
which impedes the study of the Fourier expansion in Sobolev orthogonal polynomials by using the standard techniques.
Indeed, except one class of Sobolev orthogonal polynomials, little has been done in this direction.

The exceptional class is tied with applications in the spectral methods for solving differential equations. Approximations
using orthogonal expansions in Sobolev spaces have been extensively studied by spectral methods community;
see, for example, \cite{CaQu} and numerous books (for example, \cite{Boyd, CHQZ, HGG}) on the subject,
although Sobolev orthogonal polynomials are not often used or used only implicitly. As an example of the exceptional
class, consider the inner product
$$
  \la f, g\ra_{-s} : = \int_{-1}^1 f^{(s)} (x) g^{(s)}(x)  dx +
    \sum_{j=0}^{s-1} \lambda_j [f^{(j)}(1) + f^{(-1)} (-1)].
$$
The orthogonal polynomials with respect to this inner product have been used implicitly in the spectral method;
see, for example, \cite{GW04, GSW09, SWL}. The error estimate for the spectral method requires estimating
the error of polynomial approximation in the Sobolev space $W_2^s[-1,1]$. By subtracting a polynomial
$g \in \Pi_{2s-1}$ that satisfies $g^{(j)}(1) = g^{(j)}(-1) =0$ for $j =0,1,\ldots, m-1$, we can assume that $f$
vanishes on the boundary terms. Let $S_n^{-s} f$ and $S_n^0$ denote the $n$-th partial sum of the Fourier expansion
of $f$ in terms of the Sobolev orthogonal polynomials and the Legendre polynomials $P_n^{(0,0)}$, respectively. It is not
difficult to see that
$
   \partial^s S_n^{-s} f = S_{n-s}^0 \partial^s f,
$
which implies that $\partial^s (f - S_n^{-s}) = \partial^s f -S_{n-s}^0 ( \partial^s f)$. Thus, the estimate for the highest order
derivative in $W_2^s$ follows from the estimate of the derivative of $f$ in the usual $L^2$ norm. The estimate
that the spectral method requires is of the form
$$
    \|f - p_n \|_{W_2^s} \le c n^{- r + s} \|f\|_{W_2^r}, \qquad r \ge s.
$$
where $p_n$ is a polynomial of degree $n$. Under the name of simultaneous approximation, an inequality of this type
has also been established in approximation theory for $W_p^s$ with $1 \le p < \infty$; see, for example, \cite{KK}.
Such results can also be established for the inner product with the Jacobi weight function (see \cite{SWL} and
the references therein).

In the same spirit, spectral approximation on the unit ball was studied in \cite{LX13}, where the Sobolev
orthogonal polynomials associated with \eqref{eq:ipd-s} play an essential role. In contrast to one variable, we can
no longer construct a polynomial $g$ of low degree so that $f-g$ vanishes on the boundary and, as a result,
the analysis in \cite{LX13} is more involved.

Besides the above class, little has been established for Fourier expansions in other Soboleve orthogonal polynomials.
For the Legendre-Sobolev inner product
$$
  \la f,g\ra_\lambda = \int_{-1}^1 f(x) g(x) dx + \lambda \int_{-1}^1 f'(x) g'(x) dx, \qquad \lambda \ge 0,
$$
which is \eqref{eq:ipdLeg1} and the first Sobolev inner product ever studied, it was observed in
\cite{ISKN}, based on computational evidence, that the Fourier expansion in the Sobolev orthogonal
polynomials behaviors remarkably better than the Fourier expansion in the Legendre toward the
end point of the interval. Similar phenomenon was also observed in later papers on Sobolev orthogonal
polynomials associated with measures that are coherent pairs, see \cite{JMPP14} and the references therein.
It should be mentioned that computation of Sobolev orthogonal polynomials is discussed in \cite{G, GZ}, where
several algorithms are provided.
From a heuristic point of view, Sobolev orthogonal expansions should have a better approximation
behavior than ordinary orthogonal expansions, as the former requires more information on the function being
expanded. However, at the time of this writing, there has been no theoretical result that either prove or quantify
that this is indeed the case.

For approximation in Sobolev spaces, it is to be expected that one should use Sobolev orthogonality instead of
ordinary orthogonality. Given the amount of works that have been carried out over years, it is surprising how
little has been done on the Fourier expansions in Sobolev orthogonal polynomials. As discussed in Section 2,
the initial motivation for studying Sobolev orthogonal polynomials came from the problem of least square
approximation in Sobolev spaces. It is time to go back to the beginning and study the Fourier expansions
and approximation by polynomials in Sobolev spaces. We have gained substantial knowledges on the Sobolev
orthogonality, it is now time to find connections and apply what we have learnt to solve problems in other fields.
We end this survey with this call of action.

\medskip
\noindent
{\bf Acknowledgment.} 
The project was carried out when the second author was on sabbatical from the University of Oregon
and visited the Carlos III University of Madrid under its generous Excellence Chair Program. He thanks the 
first author and the Department of Mathematics in the Carlos III University for the hospitality that he received.


\begin{thebibliography}{99}


\bibitem{AX13}
         R. Akta\c{s} and Y. Xu,
         Sobolev orthogonal polynomials on a simplex,
         {\it Int. Math. Res. Notice}, 2013 (2013), no. 13, 3087--3131.

\bibitem{ALR96}
         M. Alfaro, G. L\'opez Lagomasino, M. L. Rezola,
         Some properties of zeros of Sobolev-type orthogonal polynomials,
         \textit{J. Comput. Appl. Math.}\textbf{ 69 }(1996), 171--179.

\bibitem{AlMarRezRon}
         M. Alfaro, F. Marcell\'an, M. L. Rezola, and A. Ronveaux,
         On orthogonal polynomials of Sobolev type: Algebraic properties and zeros,
         \textit{SIAM J. Math. Anal.} \textbf{23} (1992), 737--757.

\bibitem{AlMarRez95}
         M. Alfaro, F. Marcell\'an,  M. L. Rezola, and A. Ronveaux,
         Sobolev type  orthogonal polynomials: The nondiagonal case,
         \textit{J. Approx. Theory} 83 (1995), 737--757.

\bibitem{AAR02}
        M, Alfaro,  M. \'A. de Morales and M. L. Rezola,
        Orthogonality of the Jacobi polynomials with negative integer parameters,
        \textit{J. Comput. Appl. Math.} \textbf{145}, (2002), 379--386.

\bibitem{AMPR03}
        M. Alfaro, F. Marcell\'an, A. Pe\~{n}a, and M. L. Rezola,
        On linearly related orthogonal polynomials and their functionals,
        \textit{J. Math. Anal. Appl.} \textbf{287} (2003), 307--319.

\bibitem{AlMoPeRe}
         M. Alfaro, J. J. Moreno-Balc\'azar, A. Pe\~{n}a, and M. L. Rezola,
         A new approach to the asymptotics of  Sobolev type  orthogonal polynomials.
        \textit{ J. Approx. Theory} \textbf{163 } (2011), 460--480.

\bibitem{ARPP}
        M. Alfaro, T.E. P\'erez, M. A. Pi\~{n}ar, and M. L. Rezola,
        Sobolev orthogonal polynomials: the discrete-continuous case,
        \textit{Meth. Appl. Anal.}, \textbf{6} (1999), 593--616

\bibitem{Alt}
        P. Althammer,
        Eine Erweiterung des Orthogonalit\"atsbegriffes bei Polynomen und deren Anwendung auf die beste Approximation,
        \textit{J. Reine Ang. Math.} \textbf{211} (1962), 192--204.

\bibitem{AlvMor}
        R. Alvarez-Nodarse and J. J. Moreno-Balc\'azar,
        Asymptotic properties of generalized Laguerre orthogonal polynomials
        \textit{ Indag. Math. (N. S.)} \textbf{15} (2004), 151--165.

\bibitem{AH05}
       K. Atkinson and O. Hansen,
       Solving the nonlinear Poisson equation on the unit disk,
       \textit{J. Integral Equations Appl.}\textbf{17} (2005), 223--241.

\bibitem{BavMei89}
        H. Bavinck and H. G. Meijer,
        Orthogonal polynomials with respect to a symmetric inner product involving derivatives,
        \textit{Applicable Anal.}, \textbf{33} (1989), 103--117.

\bibitem{BavMei}
        H. Bavinck and H. G. Meijer,
        On Orthogonal Polynomials with respect to an inner product involving derivatives: zeros and recurrence relations,
        \textit{ Indag. Math. N. S.} \textbf{1 }(1990), 7--14.

\bibitem{BR01}
         A. C. Berti and A. Sri Ranga,
         Companion orthogonal polynomials: some applications,
         \textit{Appl. Numer. Math.} \textbf{39} (2001), 127--149.

\bibitem{BBR03}
         A. C. Berti, C. F. Bracciali, and A. Sri Ranga,
         Orthogonal polynomials associated with related measures and Sobolev orthogonal polynomials,
         \textit{Numer. Algorithms} \textbf{34} (2003), 203--216.

\bibitem{Bochner}
         S. Bochner,
         \"Uber Sturm-Liouvillesche Polynomsysteme,
         \textit{Math. Z.} \textbf{89} (1929), 730--736.

\bibitem{Boyd}
       J. P. Boyd,
       \textit{Chebyshev and Fourier Spectral Methods} (2nd edtion), Dover, New York, 2001.

\bibitem{BDFP}
	 C. F. Bracciali, A. M. Delgado, L. Fern\'{a}ndez, T. E. P\'{e}rez and M. A. Pi\~{n}ar,
	 New steps on Sobolev orthogonality in two variables,
	 \textit{J. Comput. Appl. Math}. \textbf{235} (2010), 916--926.

\bibitem{Br72}
        J. Brenner,
        \"Uber eine Erweiterung des Orthogonalt\"ats bei Polynomen,
        In \textit{Constructive Theory of Functions}, G. Alexits and S. B. Stechkin eds.
        Akad\'{e}miai Kiad\'{o}, Budapest, 1972. 77--83.

\bibitem{deBr93}
        M. G. de Bruin,
        A tool for locating zeros of orthogonal polynomials in Sobolev inner product spaces,
        \textit{J. Comput. Appl. Math}. \textbf{49} (1993), 27--35.

\bibitem{BrMe95}
        M. G. de Bruin and H. G.  Meijer,
        Zeros of orthogonal polynomials in a non-discrete Sobolev space,
        \textit{Ann. Numer. Math.} \textbf{2} (1995), 233--246.

\bibitem{CaMarMor}
        A. Cachafeiro, F. Marcell\'an and J. J. Moreno-Balc\'azar
        On asymptotic properties of Freud Sobolev orthogonal polynomials,
        \textit{J. Approx. Theory} \textbf{125} (2003) 26--41.

\bibitem{CHQZ}
       C. Canuto, M. Y. Hussaini, A. Quarteroni and T. A. Zang,
       \textit{Spectral methods. Fundamentals in single domains. Scientific Computation}.
       Springer-Verlag, Berlin, 2006.

\bibitem{CaQu}
        C. Canuto and A. Quarteroni,
        Approximation results for orthogonal polynomials in Sobolev spaces,
        \textit{Math. Comp.} \textbf{38} (1982), 67--86.

\bibitem{CD03}
        M. Castro and A. J. Dur\'an,
        Boundedness properties for Sobolev inner products,
        \textit{J. Approx. Theory} \textbf{122} (2003), 97--111.

\bibitem{Coh75}
        E. A. Cohen,
        Zero distribution and behavior of orthogonal polynomials in the Sobolev space $W^{1,2}[-1,1]$,
        \textit{SIAM J. Math. Anal.} \textbf{6} (1975), 105--116.

\bibitem{Chi78}
        T. S. Chihara,
        \textit{An Introduction to Orthogonal Polynomials}.
        Gordon and Breach, New York, 1978.

\bibitem{DaiX}
        F. Dai and Y. Xu,
        \textit{Approximation theory and harmonic analysis on spheres and balls},
        Springer Monographs in Mathematics, Springer, 2013.

\bibitem{DFPPX}
        A. M. Delgado, L. Fern\'andez, T. E. P\'erez, M. A. Pi\~{n}ar and Y. Xu,
        Orthogonal polynomials in several variables for measures with mass points.
        Numer. Algorithm 55 (2010), 245--264.

\bibitem{DPP}
       A. M. Delgado, T. E. P\'erez and M. A. Pi\~{n}ar,
       Sobolev-type orthogonal polynomials on the unit ball,
       \textit{J. Approx. Theory} \textbf{170} (2013), 94--106.

\bibitem{DM04}
         A. M. Delgado and F. Marcell\'an,
         Companion linear functionals and Sobolev inner products: A case study,
         \textit{Methods Appl. Anal. } \textbf{11} (2004), 237--266.

\bibitem{DerMar}
         M. Derevyagin and F. Marcell\'an,
         A note on the Geronimus transformation and Sobolev orthogonal polynomials,
         \textit{Numer. Algorithms}. In press.

 \bibitem{DOP09}
         C. D\'iaz Mendoza, R. Orive and H. Pijeira-Cabrera,
         Zeros and logarithmic asymptotics of Sobolev orthogonal polynomials for exponential weights,
         \textit{J. Comput. Appl. Math.}\textbf{ 233 }(2009), 691--698.


\bibitem{DX}
        C. F. Dunkl and Y. Xu,
        \textit{Orthogonal polynomials of several variables},
        Encyclopedia of Mathematics and its Applications \textbf{81}, Cambridge
        University Press, 2001; 2nd edition, 2014.

\bibitem{Du}
         A. J. Dur\'an,
         A generalization of Favard's theorem for polynomials satisfying a recurrence relation,
        \textit{J. Approx. Theory} \textbf{74} (1993), 83--109.

\bibitem{DI}
        A. J. Dur\'an and M. D. de la Iglesia,
        Differential equations for discrete Laguerre-Sobolev orthogonal polynomials,
        \textit{J. Approx. Theory}, in press. arXiv:1309.6259

\bibitem{DuVan}
         A. J. Dur\'an and W. Van Assche,
         Orthogonal matrix polynomials and higher order recurrence relations,
         \textit{Linear Algebra and Appl.} \textbf{219} (1995), 261--280.

\bibitem{EvLiMarMarkRon}
         W. D. Evans, L. L. Littlejohn, F. Marcell\'an, C. Markett and A. Ronveaux,
         On recurrence relations for Sobolev orthogonal polynomials,
         \textit{SIAM J. Math. Anal.} \textbf{26} (1995), 446--467.

\bibitem{FMPPX}
        L. Fern\'{a}ndez, F. Marcell\'{a}n, T. E. P\'{e}rez, M. A. Pi\~{n}ar and Y. Xu,
        Sobolev orthogonal polynomials on product domain, preprint, 2014.


\bibitem{G}
        W. Gautschi,
        \textit{Orthogonal Polynomials: Computation and Approximation},
        Oxford Univ. Press, 2004.

 \bibitem{GK}
        W. Gautschi and A. B. J. Kuijlaars,
        Zeros and critical points of Sobolev orthogonal polynomials,
        \textit{ J.  Approx. Theory} \textbf{91} (1997), 117--137.

\bibitem{GZ}
        W. Gautschi and M. Zhang,
        Computing orthogonal polynomials in Sobolev spaces,
        \textit{Numer. Math.} \textbf{71} (1995), 159--183.

\bibitem{GeLuMar}
        J. S. Geronimo, D.S. Lubinsky and F. Marcell\'an,
        Asymptiotic for Sobolev orthogonal polynomials for exponential weights,
        \textit{Constr. Approx.} \textbf{22} (2005),309--346.

\bibitem{Gr67}
         W. Gr\"obner,
         Orthogonale Polynomsysteme, die Gleichzeitig mit f(x) auch deren Ableitung fÕ(x) approximieren,
         in \textit{Funktionalanalysis, Approximationstheorie, Numerische Mathematik}, L. Collatz ed., ISNM 7,
         Birkh\"auser, Basel, 1967, 24--32.

\bibitem{GSW09}
         B. Guo, J. Shen and L. Wang,
         Generalized Jacobi polynomials/functions and applications to spectral methods,
         \textit{Appl. Numer. Math.}, \textbf{59} (2009), 1011--1028.

\bibitem{GW04}
         B. Guo and L. Wang,
         Jacobi approximations in non-uniformly Jacobi-weighted Sobolev spaces,
         \textit{J. Approx. Theory}, \textbf{128} (2004), 1--41.

\bibitem{HGG}
       J. S. Hesthaven, S. Gottlieb and D. Gottlieb,
       \textit{Spectral methods for time-dependent problems}.
      Cambridge Monographs on Applied and Computational Mathematics, \textbf{21},
      Cambridge Univ. Press,  Cambridge, 2007.

\bibitem{PI}
         P. Iliev,
         Krall-Jacobi commutative algebras of partial differential operators,
         \textit{J. Math. Pures Appl.}  \textbf{(9) 96} (2011),  446-461.

\bibitem{Iliev}
         P. Iliev,
         Krall-Laguerre commutative algebras of ordinary differential operators,
         \textit{Ann. Mat. Pura Appl.} \textbf{192} (2013), 203--224.

\bibitem{AKNS}
        A. Iserles, P. E.  Koch, S. P. N{\o}rsett and J. M. Sanz-Serna,
        On polynomials orthogonal with respect to certain Sobolev Inner Products,
        \textit{J. Approx. Theory} \textbf{65} (1991), 151--175.

\bibitem{ISKN}
        A. Iserles, J. M. Sanz-Serna, P. E.  Koch and S. P. Norsett,
        Orthogonality and approximation in a Sobolev space.
        in \textit{Algorithms for approximation, II (Shrivenham, 1988)}, 117--124,
        Chapman and Hall, London, 1990.

\bibitem{JP08}
        M. N. de Jesus and J. Petronilho,
        On linearly related sequences of derivatives of orthogonal polynomials,
        \textit{J. Math. Anal. Appl.} \textbf{347} (2008), 482--492.

\bibitem{JP13}
        M. N. de Jesus and J. Petronilho,
        Sobolev orthogonal polynomials and $(M,N)$ coherent pairs of measures,
        \textit{J. Comput. Appl. Math.} \textbf{237} (2013), 83--101.

\bibitem{JMPP14}
        M. N. de Jesus, F. Marcell\'an, J. Petronilho and N. Pinz\'on,
        $(M,N)$-coherent pairs of order $(m,k)$ and Sobolev orthogonal polynomials,
        J. Comput. Appl. Math. \textbf{256} (2014), 16--35.

\bibitem{KKMY02}
        D. H. Kim, K. H. Kwon, F. Marcell\'an and G. J. Yoon,
        Sobolev orthogonality and coherent pairs of moment functionals: An inverse problem,
        \textit{Internat.  Math.  J.} \textbf{2} (2002), 877--888.

 \bibitem{Koek90}
         R. Koekoek,
         Generalizations of Laguerre polynomials,
         \textit{J. Math. Anal. Appl.} \textbf{153} (1990), 576--590.

\bibitem{KKB98}
        J. Koekoek, R. Koekoek and H. Bavinck,
        On differential equations for Sobolev-type Laguerre polynomials,
        \textit{Trans. Amer. Math. Soc.} \textbf{350} (1998), 347--393.

\bibitem{Koek93}
        R. Koekoek,
        The search for differential equations for certain sets of orthogonal polynomials,
        \textit{J. Comput. Appl. Math.} \textbf{49} (1993), 111--119.

\bibitem{KM93}
        R. Koekoek and H. G. Meijer,
        A generalization of Laguerre polynomials,
        \textit{SIAM J. Math. Anal.} \textbf{24} (1993), 768--782.

\bibitem{KK}
         K. Kopotun,
         A Note on simultaneous approximation in $L_p [-1, 1]$ ($1\le p < \infty$).
         \textit{Analysis}  \textbf{15} (1995), 151--158.

\bibitem{Krall}
        A. M. Krall,
        Orthogonal polynomials satisfying fourth order differential equations,
        \textit{Proc. Royal Soc. Edinburg, Sect. A} \textbf{87} (1980/81), 271--288.

\bibitem{KL95}
        K. H. Kwon and L. L. Littlejohn,
        The orthogonality of the Laguerre polynomials $\{L_n^{(-k)}(x)\}$ for positive integers $k$,
        \textit{Ann. Numer. Anal.} \textbf{2} (1995),  289--303.

\bibitem{KL98}
        K. H. Kwon and L. L. Littlejohn,
        Sobolev orthogonal polynomials and second-order differential equations,
        \textit{Rocky Mountain J. Math.} \textit{28} (1998), 547--594.

\bibitem{KLM01}
        K. H. Kwon, J. H. Lee and F. Marcell\'an,
        Generalized coherent pairs,
        \textit{J. Math. Anal. Appl.} \textbf{253} (2001), 482--514.


\bibitem{LL06}
        J. K. Lee and L. L. Littlejohn,
         Sobolev orthogonal polynomials in two variables
         and second order partial differential equations,
         \textit{J. Math. Anal. Appl.} \textbf{322} (2006), 1001--1017.

\bibitem{Lewis}
        D.  C. Lewis.
        Polynomial least square approximations,
        \textit{Amer. J. Math.} \textbf{69} (1947),  273--278.


\bibitem{LX13}
        H. Li and Y. Xu,
        Spectral approximation on the unit ball,
        arXiv:1310.2283.

\bibitem{LMA95}
        G. L\'opez Lagomasino, F. Marcell\'an and W. Van Assche,
        Relative asymptotics for polynomials orthogonal with respect to a discrete Siobolev inner product,
        \textit{Constr. Approx.} \textbf{11} (1995), 107--137.

\bibitem{LagPi}
           G. L\'opez Lagomasino and H. Pijeira,
           $N$-th root asymptotics of Sobolev orthogonal polynomials,
           \textit{J. Approx. Theory} \textbf{99} (1999), 30--43.

\bibitem{FejHueMar}
        F. Marcell\'an, R. Xh. Zejnullahu, B. Xh. Fejzullahu and E. Huertas,
        On orthogonal polynomials with respect to certain discrete Sobolev inner product,
        \textit{ Pacific J. of Math.} \textbf{257} (2012), 167--188.

\bibitem{MMM01}
        F. Marcell\'an, A. Mart\'inez-Finkelshtein and J. Moreno-Balc\'azar,
        $k$-Coherence of measures with non-classical weights, in {\it Margarita Mathematica en Memoria
        de Jos\'e Javier Guadalupe Hern\'andez}, Servicio de Publicaciones, Universidad de la Rioja,
        Logro\~{n}o, Spain, 2001. 77--83.

\bibitem{MarMor}
        F. Marcell\'an and J. J. Moreno-Balc\'azar
        Asymptotics and zeros of Sobolev orthogonal polynomials on unbounded supports
        \textit{Acta Appl. Math.} \textbf{94} (2006), 163--192.

\bibitem{MP12}
        F. Marcell\'an and N. Pinz\'on,
        Higher order coherent pairs,
        \textit{Acta Appl. Math.} \textbf{121} (2012), 105--135.

\bibitem{MPP95}
        F. Marcell\'an, T. E. P\'erez and M. A.  Pi\~{n}ar,
        Orthogonal polynomials on weighted Sobolev spaces: the semiclassical case,
        \textit{Annals Numer. Math.} \textbf{2} (1995), 93--122.

 \bibitem{MarPePi}
        F. Marcell\'an, T. E. P\'erez and M. A.  Pi\~{n}ar,
        On zeros of Sobolev type orthogonal polynomials,
        \textit{Rend. di Mat. (Roma) Serie VII} \textbf{12} (1992), 455--473.

\bibitem{MPP96}
        F. Marcell\'an, T. E. P\'erez and M. A.  Pi\~{n}ar,
        Laguerre-Sobolev orthogonal polynomials,
        \textit{J. Comput. Appl. Math.} \textbf{71} (1996), 245--265.

\bibitem{MP95}
        F. Marcell\'an and J. C. Petronilho,
        Orthogonal polynomials and coherent pairs: The classical case,
        \textit{Indag. Math. N.S.} \textbf{6} (1995), 287--307.


 \bibitem{MarRo}
        F. Marcell\'an and  A. Ronveaux,
        On a class of polynomials orthogonal with respect to a discrete Sobolev inner product,
        \textit{ Indag. Math. N. S.} \textbf{1} (1990), 451--464.

 \bibitem{MA93}
        F. Marcell\'an and W. Van Assche,
        Relative asymptoticsfor orthogonal polynomials with a Sobolev inner product,
        \textit{J. Approx. Theory}  \textbf{72} (1993), 193--209.

\bibitem{Mar91}
        P. Maroni,
        Une th\'eorie alg\'ebrique des poly\^omes orthogonaux: Application aux polyn\^omes,
        orthogonaux semi-classiques, in \textit{Orthogonal polynomials and their applications},
        C. Brezinski et al. eds, IMACS Annals Comput. Appl. Math. \textbf{9}, Baltzer, Basel, 1991, pp 95--130.

\bibitem{Mart}
         A. Mart\'inez-Finkelshtein,
         Bernstein-Szeg\H{o}'s theorem for Sobolev orthogonal polynomials,
         \textit{Constr. Approx.} \textbf{16} (2000) 73--84.

\bibitem{MMPP}
         A. Mart\'inez-Finkelshtein, J. J. Moreno-Balc\'azar, T.E. P\'erez and M. A. Pi\~{n}ar,
         Asymptotics of Sobolev orthogonal polynomials for coherent pairs,
         J. Approx. Theory, 92 (1998), 280--293.

\bibitem{MartPij}
         A. Mart\'inez-Finkelshtein and H. Pijeira-Cabrera,
         Strong asymptotics for Sobolev orthogonal polynomials,
         \textit{J. d'Anal.  Math.} \textbf{78} (1999) 143--156.

\bibitem{Mei93}
        H. G. Meijer,
        Coherent pairs and  zeros of Sobolev-type orthogonal polynomials,
        \textit{Indag. Math. N. S.} \textbf{4} (1993), 163--176.

\bibitem{Mei94}
        H. G. Meijer,
        Sobolev orthogonal polynomials with a small number of real zeros,
        \textit{J. Approx. Theory} \textbf{77} (1994), 305--313.


\bibitem{Mei96}
        H. G. Meijer,
        A short history of orthogonal polynomials in a Sobolev space I. The non-discrete case,
        \textit{Niew Arch. voor Wiskunde} \textbf{14} (1996), 93--112.

\bibitem{Mei97}
        H. G. Meijer,
        Determination of all coherent pairs of functionals,
        \textit{J. Approx. Theory} \textbf{89} (1997), 321--343.


  \bibitem{MeideBr02}
        H. G.  Meijer and M. G. de Bruin,
        Zeros of Sobolev orthogonal polynomials following from coherent pairs,
        \textit{J. Comput. Appl. Math.}\textbf{ 139} (2002), 253--274.

\bibitem{APP98}
        M. \'A. de Morales, T.E. P\'erez and M. A. Pi\~{n}ar,
        Sobolev orthogonality for the Gegenbauer polynomials $\{C_n^{(-N+1/2)}\}_{n\ge 0}$
        \textit{J. Comput. Appl. Math.}, \textbf{100} (1998), 111--120.

\bibitem{Mor07}
        J. J. Moreno-Balc\'azar,
        A note on the zeros of Freud-Sobolev orthogonal polynomials,
        \textit{J. Comput. Appl. Math.} \textbf{207} (2007), 338--344.


\bibitem{PP96}
        T. E. P\'erez and M. A.  Pi\~{n}ar,
        On Sobolev orthogonality for the generalized Laguerre polynomials,
         \textit{J. Approx. Theory}, \textbf{86} (1996), 278--285.

\bibitem{PPX13}
        T. E. P\'erez, M. A. Pi\~{n}ar and Y. Xu,
        Weighted Sobolev orthogonal polynomials on the unit ball,
        {\it J. Approx. Theory}, 171 (2013), 84--104.

\bibitem{PX09}
        M. Pi\~nar and Y. Xu,
        Orthogonal polynomials and partial differential equations on the unit ball,
        \textit{Proc. Amer. Math. Soc.} \textbf{137} (2009), 2979--2987.

\bibitem{Sch72}
        F. W. Sch\"afke,
        Zu den Orthogonalpolynomen von Althammer,
        \textit{J. Reine Ang. Math.} \textbf{252} (1972), 195--199.

\bibitem{SW73}
        F. W. Sch\"afke and G. Wolf,
        Einfache verallgemeinerte klassische Orthogonal polynome,
        \textit{J. Reine Ang. Math.} \textbf{262/263} (1973), 339--355.

\bibitem{SWL}
        J. Shen, L. Wang and H. Li,
         A triangular spectral element method using fully tensorial rational basis functions,
        \textit{SIAM J. Numer. Anal.} \textbf{47} (2009), 1619--1650.

\bibitem{Szego}
        G. Szeg\H{o},
        \textit{Orthogonal polynomials}, 4th edition,
        Amer. Math. Soc. Colloq. Publ. \textbf{23}, Amer. Math. Soc. Providence, RI. 1975.

\bibitem{X06}
        Y. Xu,
        A family of Sobolev orthogonal polynomials on the unit ball,
        \textit{J. Approx. Theory}, \textbf{138} (2006), 232--241.

\bibitem{X08}
        Y. Xu,
        Sobolev orthogonal polynomials defined via gradient on the unit ball,
       \textit{J. Approx. Theory}  \textbf{152} (2008), 52--65.


\end{thebibliography}
 \end{document}